\documentclass[a4paper,reqno]{amsart} 

\usepackage[all]{xy}

\usepackage[latin1]{inputenc}
\usepackage{amssymb}
\usepackage{amsmath, amsthm, amssymb}
\usepackage{graphics}
\usepackage{mathbbol}
\usepackage{color}
\usepackage{hyperref}

\theoremstyle{plain}
\newtheorem{thm}{Theorem} [section]
\newtheorem{lem}[thm]{Lemma}

\newtheorem{cor}[thm]{Corollary}
\newtheorem{prop}[thm]{Proposition}
\theoremstyle{definition}
\newtheorem{dfn}[thm]{Definition}
\newtheorem{ex}[thm]{Example}
\newtheorem{rem}[thm]{Remark}

\newcommand{\NN}{\mathbb N}
\newcommand{\ZZ}{\mathbb Z}
\newcommand{\CC}{\mathbb C}
\newcommand{\udim}{\underline{\dim}}
\newcommand{\ud}{\mathbf{d}}
\newcommand{\uc}{\mathbf{c}}

\newcommand{\ut}{\mathbf{t}}
\newcommand{\uj}{\mathbf{j}}
\newcommand{\uF}{\mathbf{F}}
\newcommand{\uB}{\mathbf{B}}
\newcommand{\un}{\mathbf{n}}

 \def\End{\mathop{\rm End}}
 \def\Hom{\mathop{\rm Hom}}
 \def\Rep{\mathop{\rm rep}}
 \def\Mod{\mathop{\rm mod}}
 \def\Gr{\mathop{\rm Gr}}
 \def\F{\mathop{\mathcal F}}
\def\GL{\mathop{\rm GL}}
\def\rk{\mathop{\rm rk}}
\def\id{\text{id}}

\begin{document}

\title[Quiver Grassmannians and Ringel-Hall algebras]{Euler characteristic of quiver Grassmannians and Ringel-Hall algebras of string algebras}

\author{Nicolas Poettering}
\address{Nicolas Poettering, Mathematisches Institut, Universität Bonn, Endenicher Allee 60, D-53115 Bonn, Germany}
\email{n.poettering@gmail.com}

\thanks{Supported by BIGS-Mathematics, Bonn and Mathematisches Institut of the University Bonn}

\begin{abstract} We compute the Euler characteristics of quiver Grassmannians and quiver flag varieties of tree and band modules and prove their positivity. This generalizes some results by G.C.\ Irelli \cite{Irelli}.

As an application we consider the Ringel-Hall algebra $C(A)$ of some string algebras $A$ and compute in combinatorial terms the products of arbitrary functions in $C(A)$.
\end{abstract}

\maketitle

\subsection*{Keywords} Euler characteristic, quiver Grassmannian, quiver flag variety, quiver representation,  Ringel-Hall algebra, string algebra, band module, string module, tree module

% \newpage\tableofcontents\newpage

\section{Introduction}

\subsection{Motivation} Fomin and Zelevinsky (see \cite{FZI}, \cite{FZII}, \cite{FZIV}) have introduced cluster algebras. For their studies the Euler characteristics of a class of projective varieties, called quiver Grassmannians, are important (see \cite{CC}, \cite{DWZII}). For instance, Caldero and Keller have shown in \cite{CK1} and \cite{CK2} that the Euler characteristic plays a central role for the categorification of cluster algebras.

\subsection{Basic concepts} We use and improve a technique of Irelli \cite{Irelli} to compute Euler characteristics of such projective varieties. In general it is hard to compute the Euler characteristic of a quiver Grassmannian, but in the case of tree and band modules we show that it is only a combinatorial task.

Let $Q$ be a quiver, $M$ a finite-dimensional representation of $Q$ over $\CC$ and $\ud$ a dimension vector of $Q$. Then the\textit{ quiver Grassmannian} ${\Gr}_{\ud}(M)$ is the complex projective variety of subrepresentations of $M$ with dimension vector $\ud$ (see Definition~\ref{dfn-qui-Gra}). Our aim is to compute its Euler characteristic $\chi_{\ud}(M)$. This computation can be simplified by certain algebraic actions of the one-dimensional torus $\CC^\ast$, since the Euler characteristic of a variety equals the Euler characteristic of the subset of fixed points under certain $\CC^\ast$-actions.

 To find suitable $\CC^\ast$-actions we introduce gradings of representations of $Q$ in Section~\ref{sec-Grad}. A map $\partial\colon E\to \ZZ$  with a basis $E$ of a representation $M$ is called a \textit{grading of the representation $M$}. A grading $\partial$ of a representation $M$ induces an action of $\CC^\ast$ on the vector space $M$. If this grading $\partial$ induces also an action on some locally closed subset $X$ of ${\Gr}_{\ud}(M)$, it is called \textit{stable} on $X$ (see Definition~\ref{dfn-stable}). The linear combinations of the basis vectors with the same values under a grading $\partial$ are called \textit{$\partial$-homogeneous}. For a locally closed subset $X\subseteq\Gr_{\ud}(M)$ let \begin{align}\label{gl-X-par}X^\partial:=\Big\{U\in X\Big|U\text{ has a vector space basis, which is $\partial$-homogeneous}\Big\}.\end{align}

\begin{samepage}\begin{thm}\label{thm-Grad} Let $Q$ be a quiver, $M$ a representation of $Q$, $X\subseteq\Gr_{\ud}(M)$ a locally closed subset and $\partial$ a stable grading on $X$. Then $X^\partial$ is a locally closed subset of $\Gr_{\ud}(M)$ and the Euler characteristic of $X$ equals the Euler characteristic of $X^\partial$. If the subset $X$ is a non-empty and closed in $\Gr_{\ud}(M)$, then $X^\partial$ is also non-empty and closed in $\Gr_{\ud}(M)$.
\end{thm}\end{samepage}

This theorem can be used for more than one grading at the same time or in an iterated way. 

It is easy to see that \textit{relations} $I$ of a quiver $Q$ do not affect these studies. Let $M$ be a representation of a quiver $Q$ with relations $I$. So $M$ is also a representation of the quiver $Q$ without relations. Any subrepresentation of the representation $M$ of $Q$ is also a subrepresentation of the representation $M$ of $Q$ with the relations $I$. Thus the variety $\Gr_{\ud}(M)$ for a representation $M$ of a quiver $Q$ with relations $I$ equals the variety $\Gr_{\ud}(M)$ for a representation $M$ of a quiver $Q$ without relations.

\subsection{Tree and band modules} Some special morphisms of quivers $F\colon S\to Q$ are called \textit{windings of quivers} (for further details see Section~\ref{sec-mor-rep}). Each winding induces a functor $F_\ast\colon\Rep(S)\to\Rep(Q)$ of the categories of finite-dimensional representations and a map $\uF\colon\NN^{|S_0|}\to\NN^{|Q_0|}$ of the dimension vectors of $S$ and $Q$. If $S$ is a tree and every vector space of a representation $V$ of $S$ is one-dimensional, then the image of the representation $V$ under the functor $F_\ast$ is called a \textit{tree module} (see Definition~\ref{dfn-tree-mod}). Let $S$ be a quiver of type $\tilde A_{l-1}$ and $V=(V_i,V_a)_{i\in S_0,a\in S_1}$ a finite-dimensional representation of $S$. The representation $F_\ast(V)$ of $Q$ is called a \textit{band module} if every linear map $V_a$ is an isomorphism and $F_\ast(V)$ is indecomposable (see Definition~\ref{dfn-band-mod}).
\begin{samepage}\begin{thm}\label{thm-tree-band} Let $Q$ and $S$ be quivers and $\ud$ a dimension vector of $Q$.
\begin{enumerate}
 \item \label{part-tree-band-1} Let $F\colon S\to Q$ be a tree or a band and $V$ any $S$-representation. Then \begin{align}\label{gl-tree-band}\chi_{\ud}(F_\ast(V))={\sum}_{\ut\in\uF^{-1}(\ud)}\chi_{\ut}(V).\end{align}

 \item \label{part-tree-band-2} Let $S$ be a quiver of type $\tilde A_{l-1}$ with sources $\{i_1,\ldots,i_r\}$ and sinks $\{i'_1,\ldots,i'_r\}$, $\ut=(t_1,\ldots,t_l)$ a dimension vector of $S$, $V$ a band module of $S$ and $n:=\dim_{\CC}(V_i)$ for some $i\in S_0$. Then \begin{align}\label{gl-prop-band}
 &\chi_{\ut}\left(V\right)=\left(\prod_{k=1}^r\frac{(n-t_{i_k})!}{t_{i_k}!}\frac{t_{i'_k}!}{(n-t_{i'_k})!}\right)\left(\prod_{i=1}^l\frac1{(\varepsilon_i(t_i-t_{i+1}))!}\right)\end{align} with $0!=1$,  $r!=0$ and $\frac1{r!}=0$ for all negative $r\in\ZZ$ and $t_{l+1}=t_1$.
\end{enumerate}
\end{thm}\end{samepage}
The formula in part~\eqref{part-tree-band-2} simplifies in the following way: Let $i\in S_0$. Then 
\begin{align*}
\frac1{(\varepsilon_i(t_i-t_{i+1}))!}=\left\{\begin{array}{cl}\frac1{|t_i-t_{i+1}|!}&\text{ if }t_{s(s_i)}\leq t_{t(s_i)},\\0&\text{ if }t_{s(s_i)}> t_{t(s_i)}.\end{array}\right.
\end{align*} In other words we have computed the Euler characteristics of quiver Grassmannians of all tree and band modules and proved their positivity in this theorem (see Corollary~\ref{cor-tree} and~\ref{cor-pos}).

\subsection{Quiver flag varieties} The projective variety $\F_{\ud^{(1)},\ldots,\ud^{(r)}}(M)$ of flags of subrepresentations of $M$ with dimension vectors $\ud^{(1)},\ldots,\ud^{(r)}$ is called \textit{quiver flag variety} (see Definition~\ref{dfn-qui-fla-var}). Theorem~\ref{thm-Grad} and~\ref{thm-tree-band} can also be generalized to analogous statements for such quiver flag varieties (see Corollary~\ref{cor-tree-band}).

\subsection{Ringel-Hall algebras}
 Let $A=\CC Q/I$ be a finite-dimensional $\CC$-algebra (for further details see Section~\ref{sec-quiver-path-alg}). We can associate to $A$ the \textit{Ringel-Hall algebra} $\mathcal H(A)$ and its subalgebra $C(A)$ (for further details see Section~\ref{sec-dfn-RH}).
 Let $X\subseteq{\Rep}_{\ud}(A)$ and $Y\subseteq{\Rep}_{\uc}(A)$ be locally closed and  $\GL(\CC)$-stable subsets. To consider the multiplication of the Ringel-Hall algebra $\mathcal H(A)$ we have to compute the Euler characteristic of the following locally closed subset of ${\Gr}_{\ud}(M)$ \begin{align*}\Big\{N\in{\Gr}_{\ud}(M)\Big|N\in X,M/N\in Y\Big\}.\end{align*}

Let $\uF=\left(F^{(1)},\ldots,F^{(r)}\right)$ be a tuple of trees, $\uB=\left(B^{(1)},\ldots,B^{(s)}\right)$ a tuple of bands and $\un=(n_1,\ldots,n_s)$ a tuple of positive integers (see Definition~\ref{dfn-tree-mod} and~\ref{dfn-band-mod}). Let $l(\uF)=r$, $l(\uB)=s$ be the lengths of the tuples and \begin{align*}\Eins_{\uF,\uB,\un}(M)=\left\{ \begin{array}{ll}
 1 & \text{if }\exists \lambda_1,\ldots,\lambda_s\in\CC^\ast, M\cong\bigoplus_{i=1}^rF^{(i)}_\ast(V)\oplus\bigoplus_{i=1}^sB^{(i)}_\ast(\lambda_i,n_i), \\
 0 & \text{otherwise.}\end{array}\right.\end{align*} This defines constructible functions in $\mathcal H(A)$, which are not necessarily in $C(A)$. The gradings used in the proof of Theorem~\ref{thm-tree-band} are also stable on the Grassmannians appearing in the product $\left(\Eins_{\uF,\uB,\un}\ast \Eins_{\uF',\uB',\un'}\right)(F_\ast(V))$ for some tree or band module $F_\ast(V)$. This simplifies the calculations of the Euler characteristics of these Grassmannians.

Let $\uF=\left(F^{(1)},\ldots,F^{(r)}\right)$ with $F^{(i)}\colon S^{(i)}\to Q$ be a tuple of windings. We define a set of tuples of windings by \begin{align*}\mathcal G(\uF)=\left\{\widetilde\uF=\left(\widetilde F^{(1)},\ldots,\widetilde F^{(r)}\right)\middle|F\widetilde F^{(i)}=F^{(i)}\ \forall i\right\}.\end{align*} 
Thus for all $i$ the following diagram commutes. \begin{align*}\vcenter{\begin{xy}\SelectTips{cm}{} \xymatrix@-0.5pc{S\ar[rr]^F&&Q\\&S^{(i)}\ar@{-->}[ul]^{\widetilde F^{(i)}}\ar[ur]_{F^{(i)}}}\end{xy}}\end{align*}

\begin{samepage}\begin{thm}\label{thm-RH} Let $A=\CC Q/I$ be a finite-dimensional algebra, $F\colon S\to Q$ a tree or a band and $V$ a $S$-representation such that $F_\ast(V)$ is an $A$-module. Let $\uF$ and $\uF'$ be tuples of trees, $\uB$ and $\uB'$ tuples of bands and $\un$ and $\un'$ tuples of positive integers. Then \begin{align}\left(\Eins_{\uF,\uB,\un}\ast \Eins_{\uF',\uB',\un'}\right)(F_\ast(V))={\sum}\left(\Eins_{\widetilde\uF,\widetilde\uB,\un}\ast \Eins_{\widetilde\uF',\widetilde\uB',\un'}\right)(V), \end{align} where the sum is over all $\widetilde\uF\in\mathcal G(\uF),\widetilde\uF'\in\mathcal G(\uF'),\widetilde\uB\in\mathcal G(\uB),\widetilde\uB'\in\mathcal G(\uB')$.
\end{thm}\end{samepage}

 The functions $\Eins_{\uF,\uB,\un}$, $\Eins_{\uF',\uB',\un'}$ and the corresponding products are in $\mathcal H(A)$. The other functions $\Eins_{\widetilde\uF,\widetilde\uB,\un}$, $\Eins_{\widetilde\uF',\widetilde\uB',\un'}$ and corresponding products in this case are in $\mathcal H(\CC S)$. So this theorem shows: To calculate $\left(\Eins_{\uF,\uB,\un}\ast \Eins_{\uF',\uB',\un'}\right)(F_\ast(V))$ it is enough to consider some combinatorics and $S$-representations, where $S$ is a tree or a quiver of type $\tilde A_{l-1}$.

Actually for a string algebra $A=\CC Q/I$ (see Section~\ref{sec-dfn-string-alg}) the computation of arbitrary products of functions in $C(A)$ is reduced to a purely combinatorial task (see Corollary~\ref{cor-RH-pur-com}).

\subsection{} The paper is organized as follows: First we define our main objects in Section~\ref{sec-quiver}, then we explain our results in Section~\ref{sec-main-res}. After that we introduce the gradings as a useful tool in Section~\ref{sec-Grad} and then we prove Theorem~\ref{thm-Grad} in Section~\ref{sec-thm-Grad}. Both are used to prove Theorem~\ref{thm-tree-band}, Theorem~\ref{thm-RH} and the results of Section~\ref{sec-main-res} in the remaining sections.

\section{Main definitions}\label{sec-quiver}
\subsection{Quivers and path algebras}\label{sec-quiver-path-alg}
Let $Q=(Q_0,Q_1,s,t)$ be a \textit{quiver}, i.e.\ a finite oriented graph with vertex set $Q_0$, arrow set $Q_1$ and maps $s,t\colon Q_1\to Q_0$ indicating the start and terminal point of each arrow. A finite-dimensional \textit{representation} $M=(M_i,M_a)_{i\in Q_0,a\in Q_1}$ of $Q$ (or \textit{$Q$-representation} for short) is a set of finite-dimensional $\CC$-vector spaces $\{M_i|i\in Q_0\}$ and a set of $\CC$-linear maps $\{M_a\colon M_{s(a)}\to M_{t(a)}|a\in Q_1\}$. A \textit{morphism} $f=(f_i)_{i\in Q_0}$ of $Q$-representations from $M$ to $N$ is a set of $\CC$-linear maps $\{f_i\colon M_i\to N_i|i\in Q_0\}$ such that $f_{t(a)}M_a=N_af_{s(a)}$ for all $a\in Q_1$. Let $\Rep(Q)$ denote the \textit{category of finite-dimensional $Q$-representations.}

 A \textit{subrepresentation} $N=(N_i)_{i\in Q_0}$ of $M=(M_i,M_a)_{i\in Q_0,a\in Q_1}$ is a set of subspaces $\{N_i\subseteq M_i|i\in Q_0\}$ such that $M_a(N_{s(a)})\subseteq N_{t(a)}$ for all $a\in Q_1$. So every subrepresentation $N=(N_i)_{i\in Q_0}$ of a $Q$-representation $M=(M_i,M_a)_{i\in Q_0,a\in Q_1}$ is again a $Q$-representation by $(N_i,M_a|_{N_{s(a)}})_{i\in Q_0,a\in Q_1}$. In this case we write $N\subseteq M$. The \textit{dimension} of a $Q$-representation $M$ is $\dim(M):=\sum_{i\in Q_0}\dim_{\CC}(M_i)$ and its \textit{dimension vector} is the tuple $\udim(M):=(\dim_{\CC}(M_i))_{i\in Q_0}\in\NN^{|Q_0|}$.

 Let $Q$ be a quiver. An \textit{oriented path} $\rho=a_1\ldots a_n$ of $Q$ is the concatenation of some arrows $a_1,\ldots,a_n\in Q_1$ such that $t(a_{i+1})=s(a_i)$ for all $1\leq i<n$. Additionally we introduce a path $e_i$ of length zero for each vertex $i\in Q_0$. The \textit{path algebra} $\CC Q$ of a quiver $Q$ is the $\CC$-vector space with the set of oriented paths as a basis. The product of basis vectors is given by the concatenation of paths if possible or by zero otherwise. It is well known that the \textit{category $\Mod(\CC Q)$ of finite-dimensional $\CC Q$-modules} is equivalent to the category $\Rep(Q)$. So we can think of $Q$-representations as $\CC Q$-modules and vice versa.

Let $\NN_{>0}=\NN-\{0\}$. Let $Q$ be a quiver and $\CC Q^+$ the ideal in the path algebra $\CC Q$, which is generated by all arrows in $Q_1$ of $Q$. An ideal $I$ of the path algebra $\CC Q$ is called \textit{admissible} if a $k\in\NN_{>0}$ exists such that $(\CC Q^+)^k\subseteq I\subseteq (\CC Q^+)^2$. In this case, $A=\CC Q/I$ is a finite-dimensional $\CC$-algebra such that the isomorphism classes of simple representations are in bijection with the vertices of the quiver $Q$. Let $\Mod(A)$ be the\textit{ category of finite-dimensional $A$-modules}. Again we can think of $A$-modules as $Q$-representations and some $Q$-representations as $A$-modules.

 An expanded introduction to finite-dimensional algebras over an arbitrary field can be found in \cite{ASS}.

\subsection{Quiver Grassmannians}
\begin{dfn}\label{dfn-qui-Gra} Let $Q$ be a quiver, $M$ a $Q$-representation and $\ud$ a dimension vector. Then the closed subvariety \begin{align*}{\Gr}_{\ud}(M):=\Big\{U\subseteq M\Big|\udim(U)=\ud\Big\}\end{align*} of the classical Grassmannian is called the \textit{quiver Grassmannian}.
\end{dfn}

 Hence this is a projective complex variety, which is by \cite{Schof-general} in general neither smooth nor irreducible. We denote the Euler characteristic of a quasi-projective variety $X$ by $\chi(X)$ and the Euler characteristic of ${\Gr}_{\ud}(M)$ by $\chi_{\ud}(M)$ for short.

\begin{ex}
 Let $Q=1\stackrel a\longrightarrow2$, $M_1=M_2=\CC^2$ and $M_a\colon M_1\to M_2$ a linear map with $\rk(M_a)=1$. Then $M=(M_1,M_2,M_a)$ is a $Q$-representation such that ${\Gr}_{(1,1)}(M)$ can be described as $(\{\ast\}\times\mathbb P^1)\cup(\mathbb P^1\times\{\ast\})\subseteq\mathbb P^1\times\mathbb P^1$. This projective variety is neither smooth nor irreducible and $\chi_{(1,1)}(M)=3$.
\end{ex}

\begin{prop}[Riedtmann \cite{Riedtmann}]\label{prop-dir-sum} Let $Q$ be a quiver, $\ud$ a dimension vector and $M$ and $N$ $Q$-representations. Then 
\begin{align}
 \chi_{\ud}(M\oplus N)={\sum}_{0\leq \uc\leq\ud}\chi_{\uc}(M)\chi_{\ud-\uc}(N)
\end{align} with $(c_i)_{i\in Q_0}=\uc\leq\ud=(d_i)_{i\in Q_0}$ if and only if $c_i\leq d_i$ for all $i\in Q_0$.
\end{prop}

Thus it is enough to consider the Euler characteristic of Grassmannians associated to indecomposable representations.

\subsection{Tree and band modules}
\label{sec-mor-rep}

 Let $Q=(Q_0,Q_1,s,t)$ and $S=(S_0,S_1,s',t')$ be two quivers. A \textit{winding of quivers} $F\colon S\to Q$ (or \textit{winding} for short) is a pair of maps $F_0\colon S_0\to Q_0$ and $F_1\colon S_1\to Q_1$ such that the following holds:\begin{enumerate}
\item $F$ is a morphism of quivers, i.e.\ $sF_1=F_0s'$ and $tF_1=F_0t'$.

\item If $a,b\in S_1$ with $a\neq b$ and $s'(a)=s'(b)$, then $F_1(a)\neq F_1(b)$.

\item If $a,b\in S_1$ with $a\neq b$ and $t'(a)=t'(b)$, then $F_1(a)\neq F_1(b)$.
\end{enumerate}
   This generalizes Krause's definition of a winding \cite{Krause}. Let $V$ be a $S$-representation. For $i\in Q_0$ and $a\in Q_1$ set \begin{align*}(F_\ast(V))_i=\bigoplus_{j\in F_0^{-1}(i)}V_j\text{\qquad and\qquad}(F_\ast(V))_a=\bigoplus_{b\in F_1^{-1}(a)}V_b.\end{align*} This induces a functor $F_\ast\colon\Rep(S)\to\Rep(Q)$ and a map of dimension vectors $\uF\colon\NN^{|S_0|}\to\NN^{|Q_0|}$.

A simply connected quiver $S$ is called a \textit{tree}, i.e.\ for any two vertices in $S$ exists a unique not necessarily oriented path from one vertex to the other.
\begin{dfn}\label{dfn-tree-mod} Let $Q$ and $S$ be quivers and $F\colon S\to Q$ a winding. Let $V$ be a $S$-representation with $\dim_{\CC}(V_i)=1$ for all $i\in S_0$ and $V_a\neq0$ for all $a\in S_1$. If $S$ is a tree, then the representation $F_\ast(V)$ is called a \textit{tree module}. We call such a winding $F$ a \textit{tree}, too.
\end{dfn}
Each tree module $F_\ast(V)$ is indecomposable and described up to isomorphism uniquely by the winding $F\colon S\to Q$.
\begin{ex}
 Let $Q$, $S$ and $F$ be described by the following picture.
\begin{align*}F\colon S=\left(\vcenter{\begin{xy}\SelectTips{cm}{}\xymatrix@-0.4pc{
 1\ar[dr]_\alpha&2\ar[d]_<<\beta&3\ar[dl]_\gamma\\
  &3'\ar[r]^{\gamma'}&3''
  }\end{xy}}\right)\to Q=\left(\vcenter{\begin{xy}\SelectTips{cm}{}\xymatrix@-0.4pc{
  1\ar[rd]_\alpha&2\ar[d]^<<\beta\\&3\ar@(ur,dr)[]^\gamma}\end{xy}}\right) \end{align*} Let $V$ and $F_\ast(V)$ be described by the following pictures. \begin{align*}V=\left(\vcenter{\begin{xy}\SelectTips{cm}{}\xymatrix@-0.4pc{
 \CC\ar[dr]_1&\CC\ar[d]_1&\CC\ar[dl]_1\\
  &\CC\ar[r]^1&\CC
  }\end{xy}}\right),\quad F_\ast(V)=\left(\vcenter{\begin{xy}\SelectTips{cm}{}\xymatrix@-0.4pc{
  \CC\ar[rd]_{\left(\begin{smallmatrix} 0\\1\\0 \end{smallmatrix}\right)}&\CC\ar[d]^<<{\left(\begin{smallmatrix} 0\\1\\0 \end{smallmatrix}\right)}\\&\CC^3\ar@(ur,dr)[]^{\left(\begin{smallmatrix} 0&0&0\\1&0&0\\0&1&0 \end{smallmatrix}\right)}}\end{xy}}\right) \end{align*} Then $F\colon S\to Q$ is a tree and $F_\ast(V)$ a tree module.
\end{ex}

A quiver $S$ is called \textit{of type $A_l$} for some $l\in\NN_{>0}$ if $S_0=\{1,\ldots,l\}$ and $S_1=\{s_1,\ldots,s_{l-1}\}$ such that for all $i\in S_0$ with $i\neq l$ there exists a $\varepsilon_i\in\{-1,1\}$ with 
$s(s_i^{\varepsilon_i})=i+1$ and $t(s_i^{\varepsilon_i})=i$. (We use here the convention $s(a^{-1})=t(a)$ and $t(a^{-1})=s(a)$ for all $a\in S_1$.) Figure~\ref{fig-Al} visualizes a quiver $S$ of type $A_l$.
\begin{figure}[ht]
$$\vcenter{\begin{xy}\SelectTips{cm}{}\xymatrix@-0.5pc{1&2\ar[l]_{s_1^{\varepsilon_1}}&3\ar[l]_{s_2^{\varepsilon_2}}&\cdots\ar[l]&l-1\ar[l]&l\ar[l]_{\quad s_{l-1}^{\varepsilon_{l-1}}}} \end{xy}}$$
\caption{A quiver of type $A_l$}
\label{fig-Al}
\end{figure}

\begin{dfn}\label{dfn-string-mod} Let $Q$ and $S$ be quivers, $S$ of type $A_l$, $F\colon S\to Q$ a winding and $F_\ast(V)$ a tree module. Then $F$ is called a \textit{string} and $F_\ast(V)$ is called a \textit{string module}.
\end{dfn}

 A quiver $S$ is called \textit{of type $\tilde A_{l-1}$} for some $l\in\NN_{>0}$ if $S_0=\{1,\ldots,l\}$ and $S_1=\{s_1,\ldots,s_l\}$ such that for all $i\in S_0$ a $\varepsilon_i\in\{-1,1\}$ exists with $s(s_i^{\varepsilon_i})=i+1$ and $t(s_i^{\varepsilon_i})=i$. (We set $l+i:=i$ in $S_0$.) We draw a picture of a quiver of type $\tilde A_{l-1}$ in Figure~\ref{fig-tAl}.

\begin{dfn}\label{dfn-band-mod} Let $Q$ and $S$ be quivers, $B\colon S\to Q$ a winding and $V$ a $S$-representation. If $S$ is of type $\tilde A_{l-1}$, $V_a$ is an isomorphism for all $a\in S_1$ and $B_\ast(V)$ is indecomposable, then $B_\ast(V)$ is called a \textit{band module}. $B$ is called a \textit{band} if an indecomposable band module $B_\ast(V)$ exists.
\end{dfn}

Let $S$ be a quiver of type $\tilde A_{l-1}$, $B\colon S\to Q$ a winding and $V$ a $S$-representation with $\dim(V_i)=1$ for all $i\in S_0$ and $V_a\neq0$ for all $a\in S_1$.  The module $B_\ast(V)$ is not necessarily indecomposable. This is a well known problem, which is explained in the following examples.

\begin{ex}\label{ex-band-zerl} 
 Let $Q$ and $S$ be quivers, $S$ of type $\tilde A_{l-1}$, $B\colon S\to Q$ a winding such that no integer $r$ exists with $1\leq r<l$,  $B_1(s_i)=B_1(s_{i+r})$ and $\varepsilon_i=\varepsilon_{i+r}$ for all $1\leq i\leq l$. (We set $s_{l+i}:=s_i$ in $S_1$ and $\varepsilon_{l+i}:=\varepsilon_i$ for all $i\in S_0$.) Let $V$ be a $S$-representation such that $V_i=\CC^n$ for all $i\in S_0$, $V_{s_i}=\id_{\CC^n}$ for all $i\in S_0$ with $i\neq 1$ and the Jordan normal form of the map $V_{s_1}$ is an indecomposable Jordan matrix. Then $B_\ast(V)$ is indecomposable.\end{ex}

 \begin{ex}\label{ex-band-per}
 Let $Q$ and $S$ be quivers, $S$ of type $\tilde A_{l-1}$, $B\colon S\to Q$ a winding such that an integer $r$ as above exists. Let $V$ be a $S$-representation with $V_{s_i}$ is an isomorphism for all $i\in S_0$. Then $B_\ast(V)\cong\bigoplus_{i=1}^r M^{(i)}$ with $Q$-representations $M^{(i)}$ of dimension $\frac1r\dim(V)$.\end{ex}

\begin{rem}\label{rem-eind-band}
Let $r\in\NN_{>0}$. Using the Jordan normal form, the indecomposable modules of the polynomial ring $\CC[T,T^{-1}]$ of dimension $r$ are canonically parametrized by $\mathbb C^\ast$. Let $\varphi_r\colon\mathbb C^\ast\to\Mod(\CC[T,T^{-1}])$ describe this parametrization and $\varphi\colon\CC^\ast\times\NN_{>0}\to\Mod(\CC[T,T^{-1}])$ with $\varphi(\lambda,r)=\varphi_r(\lambda)$.

Let $B\colon S\to Q$ be a band and $\Mod(\CC[T,T^{-1}])$ the category of finite-dimensional $\CC[T,T^{-1}]$-modules. There exists a full and faithful functor $F\colon\Mod(\CC[T,T^{-1}])\to \Rep(S)$ such that for any representation $V$ in the image of this functor and any $a\in S_1$  the linear map $V_a$ is an isomorphism.

The map $\CC^\ast\times\NN_{>0}\stackrel{\varphi}\longrightarrow\Mod(\CC[T,T^{-1}])\stackrel{F}\longrightarrow\Rep(S)\stackrel{B_\ast}\longrightarrow\Rep(Q)$ is a parametrization of all band modules of the form $B_\ast(V)$. The image of $(\lambda,r)\in\CC^\ast\times\NN_{>0}$ under this map is denoted  $B_\ast(\lambda,r)$. Additional we define $B_\ast(\lambda,0)=0$ for all $\lambda\in\CC^\ast$. We remark that neither the functor $F$ nor our parametrization of band modules of the form $B_\ast(V)$ is unique.

Let $\lambda\in\CC^\ast$ and $r,s\in\NN$ with $r\geq s$. Then a surjective morphism $B_\ast(\lambda,r)\twoheadrightarrow B_\ast(\lambda,s)$ and a injective morphism $B_\ast(\lambda,s)\hookrightarrow B_\ast(\lambda,r)$ exists. Let $\varphi\colon B_\ast(\lambda,r)\to B_\ast(\lambda,s)$ be such a morphism. Then the kernel and the image of $\varphi$ are independent of $\varphi$. So for all $r,s\in\NN$ with $r\geq s$ exists a unique sub- and a unique factormodule of $B_\ast(\lambda,r)$ isomorphic to $B_\ast(\lambda,s)$.
\end{rem}

 \begin{ex}\label{ex-band} Let $Q=(\{\circ\},\{\alpha,\beta\},s,t)$,  $\lambda\in\CC^\ast$ and $B$ the band described by the following picture. \begin{align*}B\colon S=\left(\vcenter{\begin{xy}\SelectTips{cm}{}\xymatrix@-1pc{
&1\ar[dl]_\beta\ar[dr]^\alpha\\
2\ar[dr]_{\alpha'}&&3\ar[dl]^{\beta'}\\
&4}\end{xy}}\right)\to Q=\Big(\vcenter{\begin{xy}\SelectTips{cm}{}\xymatrix@-1pc{\circ\ar@(ld,lu)[]^\alpha\ar@(rd,ru)[]_\beta}\end{xy}}\Big)\end{align*} In this case we can assume that the band module $B_\ast(\lambda,3)$ can be visualized by Figure~\ref{fig-band-mod}.
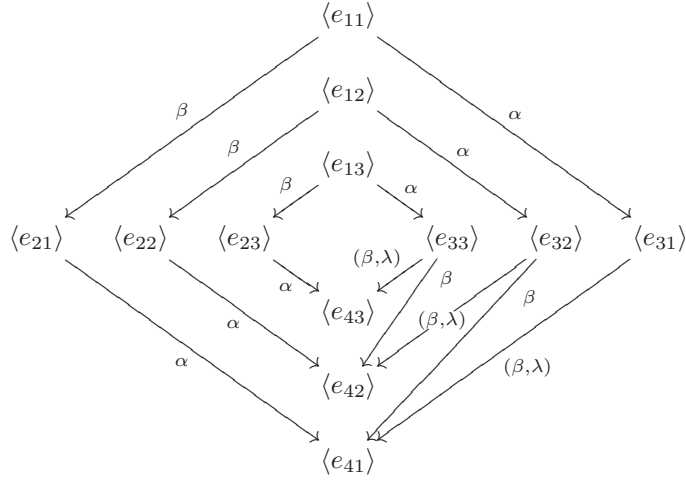
\begin{figure}[ht]
$$ \vcenter{\begin{xy}\SelectTips{cm}{}\xymatrix@-1pc{
&&&\langle e_{11}\rangle\ar[dddlll]_\beta\ar[dddrrr]^\alpha\\
&&&\langle e_{12}\rangle\ar[ddll]_\beta\ar[ddrr]^\alpha\\
&&&\langle e_{13}\rangle\ar[dl]_\beta\ar[dr]^\alpha\\
\langle e_{21}\rangle\ar[dddrrr]_\alpha&\langle e_{22}\rangle\ar[ddrr]_\alpha&\langle e_{23}\rangle\ar[dr]_\alpha&&\langle e_{33}\rangle\ar[dl]_{(\beta,\lambda)}\ar[ddl]^(.2){\beta}&\langle e_{32}\rangle\ar[ddll]|(.55){(\beta,\lambda)}\ar[dddll]^(.2){\beta}&\langle e_{31}\rangle\ar[dddlll]^{(\beta,\lambda)}\\
&&&\langle e_{43}\rangle\\
&&&\langle e_{42}\rangle\\
&&&\langle e_{41}\rangle}\end{xy}}$$
\caption{A band module $B_\ast(\lambda,3)$}
\label{fig-band-mod}
\end{figure}
\end{ex}

\begin{rem}\label{rem-ART-Atilde}
 Let $Q$ be a quiver of type $\tilde A_{l-1}$. The category $\Rep(Q)$ is well known and described in \cite{SS}. The indecomposable representations can de divided into three classes of representations: The class of preprojective representations, the class of regular representations and the class of preinjective representations. Let $M=B_\ast(\lambda,m)$ be a band module and $N=F_\ast(V)$ a string module of $Q$. Then  the following holds: \begin{itemize}
\item The band module $M$ is regular.
\item $\Hom_Q(N,M)\neq0$ and $\Hom_Q(M,N)=0$ if and only if $N$ is preprojective.
\item $\Hom_Q(N,M)=0$ and $\Hom_Q(M,N)=0$ if and only if $N$ is regular.
\item $\Hom_Q(N,M)=0$ and $\Hom_Q(M,N)\neq0$ if and only if $N$ is preinjective.
\item If $N$ is preprojective and $\dim(N)\leq\dim(M)$, then an injective map $N\hookrightarrow M$ and an indecomposable preinjective representation with dimension vector $\udim(M)-\udim(N)$ exists.
\item If $N$ is non-regular, then $N$ is determined up to isomorphism by its dimension vector.
\item If $N$ is preprojective, then all short exact sequences $0\to M\to L\to N\to0$ with some $Q$-representation $L$ split.
\item If $N$ is preinjective, then all short exact sequences $0\to N\to L\to M\to0$ with some $Q$-representation $L$ split.
\end{itemize} Let $M$ and $N$ be indecomposable preprojective $Q$-representations with $\dim(M)\geq \dim(N)$. Then $\Hom_Q(M,N)=0$ if $M\ncong N$ and all short exact sequences $0\to M\to L\to N\to0$ with some $Q$-representation $L$ split.

\end{rem}

\subsection{Ringel-Hall algebras}\label{sec-dfn-RH}

The Ringel-Hall algebras of finite-dimensional hereditary algebras over finite fields are well known objects (see \cite{Schiffmann} for an introduction). We now consider the Ringel-Hall algebra $\mathcal H(A)$ of constructible functions over a finite-dimensional $\CC$-algebra $A$. This is an idea due to Schofield \cite{Schofield}, which also appears in works of Lusztig \cite{Lusztig} and Riedtmann \cite{Riedtmann}. A nice introduction to the construction of Kapranov and Vasserot \cite{KapranovVasserot} and Joyce \cite{Joyce}, which we are using here, can be found in \cite{BridgelandLaredo}. For completeness we review the definition.

Let $A=\CC Q/I$ be a path algebra of a quiver $Q$ modulo an admissible ideal $I$. For a dimension vector $\ud$ let
 \begin{align*}{\Rep}_{\ud}(A):=\left\{(M_\alpha)_\alpha\in{\prod}_{\alpha\in Q_1}{\Hom}_\CC \left(\CC^{d_{s(\alpha)}},\CC^{d_{t(\alpha)}}\right)\middle|\left(\CC^{d_i},M_\alpha\right)_{i,\alpha}\in{\Mod}(A)\right\}\end{align*} be the \textit{module variety} of the $A$-modules with dimension vector $\ud$. The algebraic group ${\GL}_{\ud}(\CC)=\prod_{i\in Q_0}{\GL}_{d_i}(\CC)$ acts by conjugation on the variety ${\Rep}_{\ud}(A)$ such that the ${\GL}_{\ud}(\CC)$-orbits are in bijection to the isomorphism classes of $A$-modules with dimension vector $\ud$.

A function $f\colon X\to \CC$ on a variety $X$ is called \textit{constructible} if the image is finite and every fiber is locally closed. A constructible function $f\colon \Rep_{\ud}(A)\to \CC$ is called \textit{$\GL_{\ud}(\CC)$-stable} (or \textit{$\GL(\CC)$-stable} for short) if the fibers are $\GL_{\ud}(\CC)$-stable sets.

Let $\mathcal H_{\ud}(A)$ be the vector space of constructible and $\GL_{\ud}(\CC)$-stable functions on $\Rep_{\ud}(A)$. Let $\mathcal H(A)=\bigoplus_{\ud\in\NN^{|Q_0|}}\mathcal H_{\ud}(A)$ and $\ast\colon \mathcal H(A)\otimes\mathcal H(A)\to\mathcal H(A)$ with \begin{align*}(\Eins_X\ast \Eins_Y)(M)=\chi\Big(\Big\{0\subseteq N\subseteq M\Big|N\in X,M/N\in Y\Big\}\Big)\end{align*} for all $M\in\Rep_{\uc+\ud}(A)$ and all locally closed and $\GL(\CC)$-stable subsets $X\subseteq\Rep_{\ud}(A)$ and $Y\subseteq\Rep_{\uc}(A)$. For a dimension vector $\ud$ let $\Eins_{\ud}$ be the characteristic function of all representations with dimension vector $\ud$. For a $A$-module $M$ let $\Eins_{M}$ be the characteristic function of the orbit of the module $M$.
\begin{prop}  The vector space $\mathcal H(A)$ with the product $\ast$ is an associative, $\NN^{|Q_0|}$-graded algebra with unit $\Eins_0$.
\end{prop}

 Let $C(A)$ be the subalgebra of $\mathcal H(A)$ generated by the set $\left\{\Eins_{\ud}\middle|\ud\in\NN^{|Q_0|}\right\}$. The algebra $C(A)$ is a cocommutative Hopf algebra with the coproduct $\Delta\colon C(A)\to C(A)\otimes C(A)$ defined by $\Delta(f)(M,N)=f(M\oplus N)$ for all $f\in C(A)$. This is known by Joyce \cite{Joyce} and also stated in \cite{BridgelandLaredo}.

\subsection{String algebras}\label{sec-dfn-string-alg}
 Let $Q$ be a quiver and $I$ an admissible ideal. Then $A=\CC Q/I$ is called a \textit{string algebra} if the following hold:\begin{enumerate}
    \item At most two arrows start in each vertex of $Q$.
    \item At most two arrows end in each vertex of $Q$.

    \item If $i,i',j,k\in Q_0$ and $(\alpha\colon i\to j),(\beta\colon i'\to j),(\gamma\colon j\to k)\in Q_1$ with $\alpha\neq\beta$, then $\alpha\gamma\in I$ or $\beta\gamma\in I$.

    \item If $i,j,k,k'\in Q_0$ and $(\alpha\colon i\to j),(\beta\colon j\to k),(\gamma\colon j\to k')\in Q_1$ with $\beta\neq\gamma$, then $\alpha\beta\in I$ or $\alpha\gamma\in I$.
     \item The ideal $I$ is generated by oriented paths of $Q$.
\end{enumerate}
\begin{ex} Let $Q$ be as in Example~\ref{ex-band}. Then $A=\CC Q/(\alpha^2,\beta^2,\alpha\beta\alpha)$ is a string algebra and the set $\{e_\circ,\alpha,\beta, \alpha\beta,\beta\alpha,\beta\alpha\beta\}$ of pathes is a basis of the vector space $A$.

\end{ex}

 Let $A$ be a string algebra. Then it is well known that every indecomposable $A$-module is a string or a band module.

\section{Main results}\label{sec-main-res}
 In this section we explain our results in more detail.
\subsection{Tree and band modules}\label{sec-res-bb-rep} All the corollaries and examples of this section are strictly related to Theorem~\ref {thm-tree-band}.

\begin{cor}\label{cor-tree} If all vector spaces $V_i$ in Theorem~\ref{thm-tree-band}\eqref{part-tree-band-1} are one-dimensional and all the maps $V_a$ are non zero, we have to count successor closed subquivers of $S$ with dimension vectors in $\uF^{-1}(\ud)$ to compute $\chi_{\ud}(F_\ast(V))$.
\end{cor}
This corollary follows immediately from Theorem~\ref {thm-tree-band}.

\begin{cor}\label{cor-pos}
  Let $Q$ be a quiver, $M$ a tree or band module and $\ud$ a dimension vector of $Q$ such that the variety $\Gr_{\ud}(M)$ is non-empty. Then $\chi_{\ud}(M)>0$.
\end{cor}

\begin{proof} The inequality $\chi_{\ud}(M)\geq 0$ is clear by Theorem~\ref{thm-tree-band}. We prove the statement of Theorem~\ref{thm-tree-band} by applying Theorem~\ref{thm-Grad} several times. So also the stronger inequality $\chi_{\ud}(M)>0$ follows.\end{proof}

 If the quiver $S$ is an oriented cycle, each indecomposable band module $B_\ast(V)$ has a unique filtration with $n=\dim_{\CC}(V_i)$ pairwise isomorphic simple factors of dimension $|S_0|$. In this case Theorem~\ref{thm-tree-band}\eqref{part-tree-band-2} holds (see Example~\ref{ex-orient-cyc}). Therefore we can assume without loss of generality that $r\geq1$ and $1\leq i_1<i'_1<i_2<i'_2\ldots<i_r<i'_r\leq l$. The quiver $S$ is visualized in Figure~\ref{fig-tAl}.
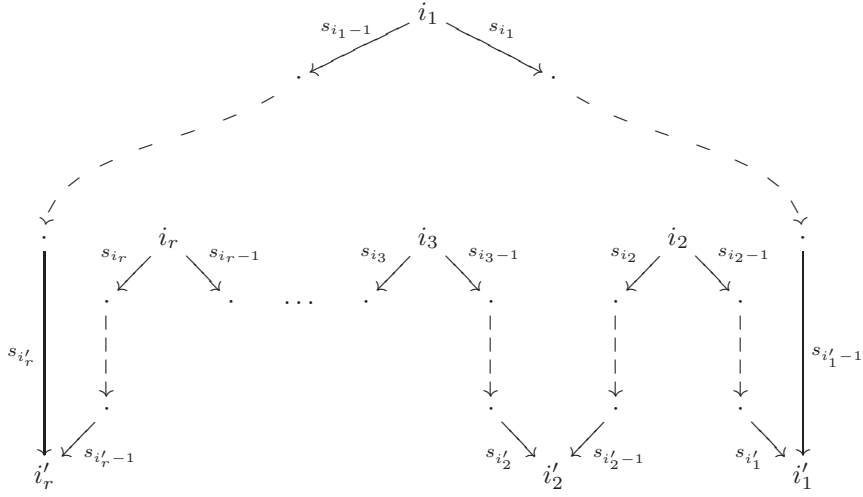
\begin{figure}[ht]
$$\vcenter{\begin{xy}\SelectTips{cm}{}\xymatrix@-1pc{
&&&&&&i_1\ar[drr]^{s_{i_1}}\ar[dll]_{s_{i_1-1}}\\
&&&&\cdot\ar@(dl,u)@{-->}[dddllll]&&&&\cdot\ar@(dr,u)@{-->}[dddrrrr]\\
\\\\
\cdot\ar[dddd]_{s_{i'_r}}&&i_r\ar[dl]_{s_{i_r}}\ar[dr]^{s_{i_r-1}}&&&&i_3\ar[dl]_{s_{i_3}}\ar[dr]^{s_{i_3-1}}&&&&i_2\ar[dl]_{s_{i_2}}\ar[dr]^{s_{i_2-1}}&&\cdot\ar[dddd]^{s_{i'_1-1}}\\
&\cdot\ar@{-->}[dd]&&\cdot&\ldots&\cdot&&\cdot\ar@{-->}[dd]&&\cdot\ar@{-->}[dd]&&\cdot\ar@{-->}[dd]\\
\\
&\cdot\ar[dl]^{s_{i'_r-1}}&&&&&&\cdot\ar[dr]_{s_{i'_2}}&&\cdot\ar[dl]^{s_{i'_2-1}}&&\cdot\ar[dr]_{s_{i'_1}}\\
i'_r&&&&&&&&i'_2&&&&i'_1} \end{xy}}$$
\caption{A quiver of type $\tilde A_{l-1}$}
\label{fig-tAl}
\end{figure}

\begin{ex}\label{ex-orient-cyc} Let $S$, $V$ and $\ut$ be as in Theorem~\ref{thm-tree-band}\eqref{part-tree-band-2}. Let $t_1=t_2=\ldots=t_l\leq n$. Then $\chi_{\ut}\left(V\right)=1$.
\end{ex}

The next example shows one result of \cite[Proposition 3]{Irelli} as a special case of Theorem~\ref{thm-tree-band}\eqref{part-tree-band-2}.

\begin{ex}Let $S$, $V$ and $\ut$ be as in Theorem~\ref{thm-tree-band}\eqref{part-tree-band-2}. Let $r=1$, $i_1=1$ and $i'_1=l$. Then \begin{align*} \chi_{\ut}\left(V\right)=\begin{pmatrix}t_l\\t_1\end{pmatrix}\begin{pmatrix}n-t_1\\n-t_l\end{pmatrix} \frac{(t_l-t_1)!}{\prod_{i=1}^{l-1}(t_{i+1}-t_i)!}.\end{align*}
\end{ex}

 \begin{ex} Let $Q$, $B$ be as in Example~\ref{ex-band} and $M=B_\ast(\lambda,2)$ a band module with $\lambda\in\CC^\ast$. Using Theorem~\ref{thm-tree-band}, it is easy to calculate the Euler characteristics $\chi_d(M)$. For instance, \begin{align*}
 \chi_4(M)&=\chi_{(0,0,2,2)}(V)+\chi_{(0,2,0,2)}(V)+\chi_{(0,1,1,2)}(V)+\chi_{(1,1,1,1)}(V)=7.\end{align*}
\end{ex}

\begin{ex}
If $F$ is a tree or a band, Theorem~\ref{thm-tree-band}\eqref{part-tree-band-1} holds for any $S$-re\-pre\-sen\-ta\-tion $V$. Let $F$ be the winding described by the following picture.
\begin{align*}F\colon S=\left(\vcenter{\begin{xy}\SelectTips{cm}{}\xymatrix@-0.3pc{
  &2\ar[d]^\beta\\1\ar[r]^\alpha&4\ar[r]^{\alpha'}&3}\end{xy}}\right)\to Q=\Big(\vcenter{\begin{xy}\SelectTips{cm}{}\xymatrix@-1pc{\circ\ar@(ld,lu)[]^\alpha\ar@(rd,ru)[]_\beta}\end{xy}}\Big) \end{align*} Let $V$ be an indecomposable $S$-representation with dimension vector $(1,1,1,2)$. Then \begin{align*} \chi_{3}\left(F_\ast(V)\right)= \chi_{(1,0,1,1)}(V)+\chi_{(0,1,1,1)}(V)+\chi_{(0,0,1,2)}(V)=3.\end{align*}
\end{ex}

\begin{ex}
 If $S$ is not a tree and not a band, Equation~\eqref{gl-tree-band} does not hold in general. To see this we consider the winding $F$ described by the following picture. \begin{align*}F\colon S=\left(\vcenter{\begin{xy}\SelectTips{cm}{}\xymatrix@-0.9pc{
  &&&3\\
  1\ar[r]^\alpha&2\ar[rru]^\beta\ar[rrd]_{\gamma'}&2'\ar[rd]^{\beta'}\ar[ru]_\gamma\\
  &&&3'
  }\end{xy}}\right)\to Q=\left(\vcenter{\begin{xy}\SelectTips{cm}{}\xymatrix@+0pc{
  1\ar[r]^\alpha&2\ar@<0.5ex>[r]^\beta\ar@<-0.5ex>[r]_\gamma&3}\end{xy}}\right) \end{align*}
 Let $V$ be a $S$-representation with $\dim_{\CC}(V_i)=1$ for all $i\in S_0$ and $V_a\neq0$ for all $a\in S_1$. Then $F_\ast(V)$ is indecomposable and \begin{align*}\chi_{(0,1,1)}\left(F_\ast(V)\right)=2\neq0={\sum}_{\ut\in\uF^{-1}((0,1,1))}\chi_{\ut}(V).\end{align*} It is easy to see that there exists no quiver $S$, no winding $F$ and no $S$-representation $V$ with $\dim_{\CC}(V_i)=1$ for all vertices $i$ such that a formula similar to Equation~\eqref{gl-tree-band} holds. So it is not possible to describe these Euler characteristics purely combinatorial using our techniques.
\end{ex}

\subsection{Quiver flag varieties}\label{sec-flag-var}

\begin{dfn}\label{dfn-qui-fla-var} Let $Q$ be a quiver, $M$ a $Q$-representation and $\ud^{(1)},\ldots,\ud^{(r)}$ dimension vectors.
 Then the closed subvariety \begin{align*}{\F}_{\ud^{(1)},\ldots,\ud^{(r)}}(M):=\left\{0\subseteq U^{(1)}\subseteq\ldots\subseteq U^{(r)}\subseteq M\middle|U^{(i)}\in{\Gr}_{\ud^{(i)}}(M)\ \forall i\right\}\end{align*} of the classical partial flag variety is called the \textit{quiver flag variety}.
\end{dfn}

 We denote the Euler characteristic of ${\F}_{\ud^{(1)},\ldots,\ud^{(r)}}(M)$ by $\chi_{\ud^{(1)},\ldots,\ud^{(r)}}(M)$. The following corollaries of Theorem~\ref{thm-tree-band} follow immediately from the analogous statements for the quiver Grassmannians.

\begin{cor}[Riedtmann]\label{cor-dir-sum} Let $Q$ be a quiver, $\ud^{(1)},\ldots,\ud^{(r)}$ dimension vectors and $M$ and $N$ $Q$-representations. Then 
\begin{align*}
 \chi_{\ud^{(1)},\ldots,\ud^{(r)}}(M\oplus N)={\sum}_{0\leq \uc^{(i)}\leq\ud^{(i)}}\chi_{\uc^{(1)},\ldots,\uc^{(r)}}(M)\chi_{\ud^{(1)}-\uc^{(1)},\ldots,\ud^{(r)}-\uc^{(r)}}(N).
\end{align*}
\end{cor}

\begin{cor}\label{cor-tree-band} Let $Q$ and $S$ be quivers, ${\ud^{(1)},\ldots,\ud^{(r)}}$ dimension vectors of $Q$ and $V$ a $S$-representation. Let $F\colon S\to Q$ be a tree or a band. Then \begin{align*}\chi_{\ud^{(1)},\ldots,\ud^{(r)}}(F_\ast(V))={\sum}_{\ut^{(i)}\in\uF^{-1}(\ud^{(i)})}\chi_{\ut^{(1)},\ldots,\ut^{(r)}}(V).\end{align*}
In particular if all vector spaces $V_i$ are one-dimensional and all the maps $V_a$ are non zero, we have to count flags of successor closed subquivers of $S$ with dimension vectors in $\uF^{-1}(\ud^{(i)})$ to compute $\chi_{\ud^{(1)},\ldots,\ud^{(r)}}(F_\ast(V))$.
\end{cor}

\begin{ex}\label{ex-fla-Kron}
 Let $Q=(1\rightrightarrows2)$, $n\in\NN$ with $n\geq3$ and $M$ an indecomposable module with dimension $2n$. Then \begin{align*}\chi_{(1,2),(2,3)}\left(M\right)
 =8(n-2).\end{align*} A detailed proof of this equation is given in Section~\ref{sec-flag}. For this calculation it is enough to count flags of successor closed subquivers of the quiver
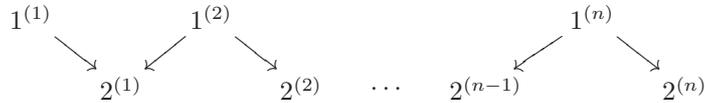
\begin{align*}
\vcenter{\begin{xy}\SelectTips{cm}{}\xymatrix@-1pc{
  1^{(1)}\ar[dr]&&1^{(2)}\ar[dl]\ar[dr]&&&&1^{(n)}\ar[dl]\ar[dr]\\
  &2^{(1)}&&2^{(2)}&\ldots&2^{(n-1)}&&2^{(n)}
  }\end{xy}}\end{align*} associated to the dimension vectors $(1,2)$ and $(2,3)$.
\end{ex}

\begin{cor}
  Let $Q$ be a quiver, $M$ a tree  module and ${\ud^{(1)},\ldots,\ud^{(r)}}$ dimension vectors of $Q$ such that $\Gr_{\ud^{(1)},\ldots,\ud^{(r)}}(M)$ is non-empty. Then $\chi_{\ud^{(1)},\ldots,\ud^{(r)}}(M)>0$.
\end{cor}

\subsection{Ringel-Hall algebras}

 We are studying the products of functions of the form $\Eins_{\uF,\uB,\un}$ in $\mathcal H(A)$. Using Section~\ref{sec-red-ind}, it is enough to consider the images of indecomposable $A$-modules. Theorem~\ref{thm-RH} shows: To compute $\left(\Eins_{\uF,\uB,\un}\ast \Eins_{\uF',\uB',\un'}\right)(F_\ast(V))$ with a tree or band $F\colon S\to Q$ we can consider some combinatorics and the products $\left(\Eins_{\widetilde\uF,\widetilde\uB,\un}\ast \Eins_{\widetilde\uF',\widetilde\uB',\un'}\right)(V)$, where $S$ is a tree or a quiver of type $\tilde A_{l-1}$.

\begin{prop}\label{prop-RH-1}  Let $A$ be a finite-dimensional algebra, $\uF$ and $\uF'$ be tuples of trees, $\uB$ and $\uB'$ tuples of bands and $\un$ and $\un'$ tuples of positive integers.
\begin{enumerate} 
\item\label{part-RH-1} Let $F_\ast(V)$ be a tree module of $A$ such that \begin{align*}\left(\Eins_{\uF,\uB,\un}\ast \Eins_{\uF',\uB',\un'}\right)(F_\ast(V))\neq0.\end{align*}
 Then $l(\uB)=l(\uB')=0$.
 \item\label{part-RH-2}  Let $B_\ast(\lambda,m)$ be a band module of $A$ with $\lambda\in\CC^\ast$ such that \begin{align*}\left(\Eins_{\uF,\uB,\un}\ast \Eins_{\uF',\uB',\un'}\right)(B_\ast(\lambda,m))\neq0.\end{align*} Then $\uB,\uB'\in\{0,(B)\}$,   $\uF$ and $\uF'$ are tuples of strings and $l(\uF)=l(\uF')$.
\end{enumerate}
\end{prop}

By Theorem~\ref{thm-RH} and this proposition, the calculation of the image of a tree module under a product $\Eins_{\uF,\uB,\un}\ast \Eins_{\uF',\uB',\un'}$ is a purely combinatorial task. It is enough to count suitable successor closed subquivers of $S$ to calculate  $\left(\Eins_{\uF}\ast \Eins_{\uF'}\right)(F_\ast(V))$.

\begin{ex}
 Let $F$ be the string described by the following picture. \begin{align*}F\colon\left(\vcenter{\begin{xy}\SelectTips{cm}{}\xymatrix@-1pc{
  &1\ar[dl]_\alpha\ar[dr]^{\beta}&&&1'\ar[dl]_{\beta'}\ar[dr]^{\alpha'}\\
  2&&3\ar[r]^\gamma&3'&&2'
  }\end{xy}}\right)\to
Q=\left(\vcenter{\begin{xy}\SelectTips{cm}{}\xymatrix@-1pc{
  &1\ar[dl]_\alpha\ar[dr]^\beta\\
  2&&3\ar@(ur,dr)[]^\gamma}\end{xy}}\right)\end{align*} Let $\uF=(2\to Q,(3\stackrel\gamma\to3')\to Q)$ and $\uF'=(1\to Q,(1\stackrel\alpha\to2)\to Q)$. Then $\left(\Eins_{\uF}\ast \Eins_{\uF'}\right)(F_\ast(V))=2$ by counting suitable subquivers.
\end{ex}

\begin{prop}\label{prop-RH-2} Let $Q$ be a quiver of type $\tilde A_{l-1}$, $\uF$ and $\uF'$ be tuples of strings, $B\colon Q\to Q$ the identity winding, $m\in\NN$ and $\lambda\in\CC^\ast$.
\begin{enumerate} 
 \item\label{part-RH-3} Let $n,n'\in\NN$ with $n+n'\leq m$. Then \begin{align}\label{gl-lem-RH-B}\left(\Eins_{\uF,B,n}\ast \Eins_{\uF',B,n'}\right)(B_\ast(\lambda,m))&=\left(\Eins_{\uF}\ast \Eins_{\uF'}\right)(B_\ast(\lambda,m-n-n')).\end{align}

 \item\label{part-RH-4} Let $n\in\NN$,  $F$ a string and $\uF(n)=(F,\ldots,F)$ with $l\left(\uF(n)\right)=n$ such that $F_\ast(V)$ and $F^{(i)}_\ast(V)$ are preprojective, $\dim\left(F_\ast(V)\right)\geq\dim\left(F^{(i)}_\ast(V)\right)$ and $F_\ast(V)\ncong F^{(i)}_\ast(V)$ for all $i$. Then
 \begin{align}&\left(\Eins_{\uF(n)\dot\cup\uF}\ast \Eins_{\uF'}\right)(B_\ast(\lambda,m))\\
 \nonumber=&{\sum}_{k_1,\ldots,k_n\in\NN}\left(\Eins_{\uF}\ast \Eins_{\uF'}\right)\left(B_\ast\left(\lambda,m-{\sum}_{i=1}^nk_i\right)\oplus{\bigoplus}_{i=1}^n I_{k_i}\right)
\end{align} with $I_{k_i}$ is an indecomposable module and $\udim(I_{k_i})=\udim(B_\ast(\lambda,k_i))-\udim(F_\ast(V))$ for all $i$.
\end{enumerate}
\end{prop}

If $\udim(B_\ast(\lambda,k_i))-\udim(F_\ast(V))>0$, the module $I_k$ exists, is preinjective and determined up to isomorphism uniquely by Remark~\ref{rem-ART-Atilde}.

Let $Q$ be a quiver of type $\tilde A_{l-1}$, $\uF''$ and $\uF'$ be tuples of strings, $B\colon Q\to Q$ the identity winding, $m\in\NN$ and $\lambda\in\CC^\ast$ such that $\left(\Eins_{\uF''}\ast \Eins_{\uF'}\right)(B_\ast(\lambda,m))\neq0$. Without loss of generality we can assume that $\dim\left(F''^{(1)}_\ast(V)\right)\geq\dim\left(F''^{(i)}_\ast(V)\right)$ for all $i$. Then $F''^{(i)}_\ast(V)$ is preprojective for all $i$ and we can apply Theorem~\ref{prop-RH-2}\eqref{part-RH-4} with $F=F''^{(1)}$ and $\uF=\{F''^{(i)}|F''^{(i)}(V)\ncong F(V)\}$. Thus the following corollary follows.

\begin{cor}\label{cor-RH-pur-com} Let $A=\CC Q/I$ be a finite-dimensional algebra, $M$ a direct sum of tree and a band modules of $Q$ such that $M$ is an $A$-module. Let $\uF$ and $\uF'$ be tuples of trees, $\uB$ and $\uB'$ tuples of bands and $\un$ and $\un'$ tuples of positive integers. Then $\Eins_{\uF,\uB,\un}\ast \Eins_{\uF',\uB',\un'}(M)$ can be computed combinatorially.
\end{cor}

\subsection{String algebras} In this section we consider the Ringel-Hall algebras of string algebras.

\begin{cor}\label{cor-RH-F-B}  Let $A$ be a string algebra. Let $\uF$ be a tuple of strings, $\uB$ a tuple of bands and $\un$ a tuple of positive integers. Then
\begin{align*} \Eins_{\uF}\ast \Eins_{\uB,\un}=\Eins_{\uF,\uB,\un}=\Eins_{\uB,\un}\ast \Eins_{\uF}.\end{align*}
\end{cor}
\begin{ex}\label{ex-RH-Kron}
 Let $Q=(1\rightrightarrows2)$, $\uF$ and $\uF'$ tuples of strings, $B\colon Q\to Q$ the identity winding and $m\in\NN$ such that $\left(\Eins_{\uF}\ast \Eins_{\uF'}\right)(B_\ast(\lambda,m))\neq0$. Then \begin{align}\left(\Eins_{\uF}\ast \Eins_{\uF'}\right)(B_\ast(\lambda,m))=\frac{l(\uF)!}{\prod_{\{F^{(i)}|i\}/\cong}|[F^{(i)}]|!}\frac{l(\uF')!}{\prod_{\{F'^{(i)}|i\}/\cong}|[F'^{(i)}]|!}\label{gl-RH-Kron},\end{align} where $\{F^{(i)}|i\}/\cong$ is the set of isomorphism classes and $|[F^{(i)}]|$ is the number of elements in the isomorphism class of $F^{(i)}$. For instance,
\begin{align*}\left(\Eins_{S_2^{m-r}\oplus P(S_1)^r}\ast \Eins_{S_1^{m-s}\oplus I(S_2)^s}\right)(B_\ast(\lambda,m+r+s))=\begin{pmatrix}m\\r\end{pmatrix}\begin{pmatrix}m\\s\end{pmatrix}
\end{align*} with $m,r,s\in\NN$, $S_i\in\Mod(A)$ is the simple representation associated to the vertex $i\in Q_0$, $P(S_i)\in\Mod(A)$ is the projective cover of $S_i$ and $I(S_i)\in\Mod(A)$ is the injective hull of $S_i$.
\end{ex}

\begin{cor}\label{cor-RH-alg} Let $A$ be a string algebra. Then every function in $C(A)$ is a linear combination of functions of the form $\Eins_{\uF,\uB,\un}$ with some tuple $\uF$ of strings, some tuple $\uB$ of bands and some tuple $\un$ of positive integers.
\end{cor}

 The proofs of these corollaries and of Equation~\eqref{gl-RH-Kron} are stated in Section~\ref{sec-string-alg}. In general it is much harder to give an explicit formula for $\left(\Eins_{\uF}\ast \Eins_{\uF'}\right)(B_\ast(\lambda,m))$.

\section{Gradings}\label{sec-Grad}

\subsection{Definitions} Let $Q$ be a quiver and $M=(M_i,M_a)_{i\in Q_0,a\in Q_1}$ a $Q$-representa-tion. Let $I=\{1,2,\ldots,\dim(M)\}$ and $E=\{e_j|j\in I\}$ be a basis of $\bigoplus_{i\in Q_0}M_i$ such that $E\subseteq\bigcup_{i\in Q_0}M_i$.

\begin{dfn} A map $\partial\colon E\to\ZZ$ is called a  \textit{grading} of $M$.
\end{dfn}
So every grading depends on the choice of a basis $E$. It is useful to change the basis during calculations. A vector $m=\sum_{j\in I} m_je_j\in M$ with $m_j\in\CC$ is called \textit{$\partial$-homogeneous} of degree $n\in\ZZ$ if $\partial\left(e_j\right)=n$ for all $j\in I$ with $m_j\neq0$. If $m\in M$ is $\partial$-homogeneous of degree $n\in\ZZ$, we set $\partial(m)=n$.

The following grading has been studied by Riedtmann \cite{Riedtmann}: Let $M=\bigoplus_{k=1}^r N_k$, where $N_k$ is a subrepresentation of $M$ for all $k$ and $E\subseteq\bigcup_{k=1}^r N_k$. Then the grading $\partial\colon E\to\ZZ$ with $\partial\left(e_j\right)=k$ if $e_j\in N_k$ is called \textit{Riedtmann grading} (or \textit{R-grading} for short).

\begin{dfn} Let $\partial$ and $\partial_1,\ldots,\partial_r$ be gradings of $M$ and $\Delta({\mathbf y},{\mathbf z},a)\in\ZZ$ for all $\mathbf y,\mathbf z\in\ZZ^r$ and $a\in Q_1$ such that \begin{align}\label{gl-nice} \Delta\Big((\partial_m(e_j))_{1\leq m\leq r},(\partial_m(e_i))_{1\leq m\leq r},a\Big)=\partial(e_i)-\partial(e_j) \end{align} for all $i,j\in I$ and $a\in Q_1$ with $e_i\in M_{t(a)}$, $e_j\in M_{s(a)}$ and $m_i\neq 0$ for $M_a(e_j)=\sum_{k\in I}m_ke_k$. Then $\partial$ is called a \textit{nice} $\partial_1,\ldots,\partial_r$-grading.
\end{dfn}

\begin{ex} In this example we state two extreme cases of gradings.
\begin{itemize}
\item Let $\partial$ and $\partial'$ be gradings such that $\partial'\colon E\to \ZZ$ is an injective map. Then $\partial$ is a nice $\partial'$-grading.

\item Let $\partial$ be a grading such that $\partial(e_i)=\partial(e_j)$ for all $i,j\in I$. Then $\partial$ is a nice grading.
\end{itemize}

\end{ex}

The definition of $\partial_1,\ldots,\partial_r$-nice gradings generalizes the gradings introduced by Irelli \cite{Irelli}. He only considers the nice $\emptyset$-gradings, i.e.\ $r=0$. (We say \textit{nice grading} for short.) Now we can successively apply these gradings.

By the following remark, we describe a way to visualize a nice $\partial_1,\ldots,\partial_r$-grading $\partial$ of some representations of the form $F_\ast(V)$.

\begin{rem}\label{rem-basis-FV} Let $Q$ and $S$ be quivers, $F\colon S\to Q$ a winding, $V$ a $S$-representation with $\dim_{\CC}(V_i)=1$ for all $i\in S_0$ and $V_a\neq0$ for all $a\in S_1$. Let $M=F_\ast(V)$ and $\{f_i\in V_i|i\in S_0\}$ be a basis of $\bigoplus_{i\in S_0} V_i$. Then $E:=\{F_\ast(f_i)|i\in S_0\}$ is a basis of $\bigoplus_{i\in Q_0}M_i$.\begin{itemize}

 \item Now we can illustrate each grading $\partial\colon E\to\ZZ$ of $M$ by a labeling of the quiver $S$. For this we extend $\partial$ to $E\cup S_0$ by $\partial(i)=\partial(F_\ast(f_i))$ for each $i\in S_0$.

 \item For each nice  $\partial_1,\ldots,\partial_r$-grading $\partial$ we can further extend $\partial$ in a meaningful way to $E\cup S_0\cup S_1$ by \begin{align*}\partial(a)=\Delta\Big(\big(\partial_m(s(a))\big)_{1\leq m\leq r},\big(\partial_m(t(a))\big)_{1\leq m\leq r},F_1(a)\Big)\end{align*} for all $a\in S_1$. Then by Equation~\eqref {gl-nice} \begin{align}\label{gl-nice-2}\partial(a)=\partial(t(a))-\partial(s(a))\end{align} holds for all $a\in S_1$.

 \item Let $\partial\colon S_0\cup S_1\to\ZZ$ be a map with the following conditions: 
\begin{itemize}
 \item[(S1)] The Equation~\eqref {gl-nice-2} holds for all $a\in S_1$.
 \item[(S2)] $\partial(a)=\partial(b)$ for all $a,b\in S_1$ with $F_1(a)=F_1(b)$,  $\partial_m(s(a))=\partial_m(s(b))$ and $\partial_m(t(a))=\partial_m(t(b))$ for all $m$.
\end{itemize} Then the map $\partial$ induces a nice $\partial_1,\ldots,\partial_r$-grading $\partial\colon E\to\ZZ$ on $M$.

 \item Let $S$ be a tree and $\partial\colon S_1\to\ZZ$ be a map such that the condition (S2) holds. Then the map $\partial$ induces a nice $\partial_1,\ldots,\partial_r$-grading $\partial\colon E\to\ZZ$ on $M$.

 \item Let $\partial\colon S_1\to\ZZ$. If $S$ is connected, such an induced grading $\partial$ is unique up to shift.
\end{itemize}
\end{rem}

\begin{ex}\label{ex-rem-basis-FV} Let $F_\ast(V)$ be the tree module described by the following picture. \begin{align*}F\colon S=\left(\vcenter{\begin{xy}\SelectTips{cm}{}\xymatrix@-0.8pc{
  1\ar[d]_\alpha&&1'\ar[d]^{\beta'}\\
  2\ar[dr]_\beta&&2'\ar[dl]^{\alpha'}\\
  &3
  }\end{xy}}\right)\to Q=\Big(\vcenter{\begin{xy}\SelectTips{cm}{}\xymatrix@+0pc{
  \circ\ar@(ul,dl)[]_\alpha\ar@(ur,dr)[]^\beta}\end{xy}}\Big) \end{align*} Then $F_\ast(V)$ has a basis $E=\{F_\ast(f_1),F_\ast(f_{1'}),F_\ast(f_2),F_\ast(f_{2'}),F_\ast(f_3)\}$ as above. Let $\partial_1\colon S_1\to\ZZ, \gamma\mapsto 1$ for all $\gamma\in S_1$ and $\partial_1(F_\ast(f_1))=0$. This induces by the previous remark a unique nice grading $\partial_1$ of $F_\ast(V)$. Let $\partial_2\colon S_1\to\ZZ, \beta\mapsto 1,\gamma\mapsto 0$ for all $\beta\neq\gamma\in S_1$ and $\partial_2(F_\ast(f_1))=0$ . This induces a unique nice $\partial_1$-grading $\partial_2$ of $F_\ast(V)$. So for instance $\partial_1(F_\ast(f_3))=2$, $\partial_2(F_\ast(f_3))=1$, $\partial_1(F_\ast(f_{1'}))=0$ and $\partial_2(F_\ast(f_{1'}))=1$.
\end{ex}

 Let $Q$ be a quiver, $M$ a $Q$-representation and $\partial$ a grading. The algebraic group $\CC^\ast$ acts by \begin{align}\varphi_\partial\colon\CC^\ast\to{\End}_{\CC}(M),\ \varphi_\partial(\lambda)(e_j):=\lambda^{\partial\left(e_j\right)}e_j\end{align} on the vector space $M$. This defines in some cases a $\CC^\ast$-action on the quiver Grassmannian $\Gr_{\ud}(M)$.

\subsection{Stable gradings}
 \begin{dfn}\label{dfn-stable} Let $X\subseteq\Gr_{\ud}(M)$ be a locally closed subset and $\partial$ a grading of $M$. If for all $U\in X$ and $\lambda\in\CC^\ast$ the vector space $\varphi_\partial(\lambda)U$ is in $X$, then the grading $\partial$ is called \textit{stable}.\end{dfn}

 Let $X\subseteq\Gr_{\ud}(M)$ be a locally closed subset and $\partial_1,\ldots,\partial_r$ gradings. Let \begin{align}\label{gl-X-par-1r} X^{\partial_1,\ldots,\partial_r}:=\Big\{U\in X\Big|U\text{ has a basis, which is $\partial_i$-homogeneous for each }i\Big\}.\end{align} This equation is a generalization of Equation~\eqref{gl-X-par}. By definition, each stable grading on $X$ is also a stable grading on  $X^{\partial_1,\ldots,\partial_r}$.

\begin{lem}\label{lem-stab-homo}Let $Q$ be a quiver, $M$ a $Q$-representation and $\ud$ a dimension vector.
Let $U\in\Gr_{\ud}(M)$ and $\partial_1,\ldots,\partial_r$ gradings. Then $U\in \Gr_{\ud}(M)^{\partial_1,\ldots,\partial_r}$ if and only if $\varphi_{\partial_i}(\lambda)U=U$ as vector spaces for all $i$ and $\lambda\in\CC^\ast$.
\end{lem}

\begin{proof}
 If $U\in\Gr_{\ud}(M)$ has a basis, which is $\partial_i$-homogeneous for each $i$, we get $\varphi_{\partial_i}(\lambda)U=U$ for each $i$ and $\lambda\in\CC^\ast$.

 Let $U\in\Gr_{\ud}(M)$ such that $\varphi_{\partial_i}(\lambda)U=U$ for all $i$ and $\lambda\in\CC^\ast$. Our aim is to find a basis for $U$, which is $\partial_i$-homogeneous for each $i$. Let $s\in\NN$ with $1\leq s\leq r$ and $\{m_1,\ldots,m_t\}$ be a basis of $U$, which is $\partial_i$-homogeneous for each $i$ with $1\leq i<s$. For each $j$ with $1\leq j\leq t$ let $m_j=\sum_{i\in I}\lambda_{ij}e_i$ with $\lambda_{ij}\in\CC$. For each $z\in\ZZ$ and $j\in\NN$ with $1\leq j\leq t$ define $m_{j,z}:=\sum_{i\in I,\partial(e_i)=z}\lambda_{ij}e_i\in M$. Then $m_{j,z}$ is $\partial_i$-homogeneous for each $i$ with $1\leq i\leq s$, $\varphi_{\partial_s}(\lambda)(m_{j,z})=\lambda^z m_{j,z}$ for all $\lambda\in\CC^\ast$ and $m_j=\sum_{z\in\ZZ}m_{j,z}$. Then $\varphi_{\partial_s}(\lambda)(m_j)=\sum_{z\in\ZZ}\lambda^z m_{j,z}\in U$ for all $\lambda\in\CC^\ast$ and so $m_{j,z}\in U$ for all $z\in\ZZ$ and all $j$. Since $\{m_{j,z}|1\leq j\leq t,z\in\ZZ\}$ generates $U$, a subset of this set is a basis of the vector space $U$, which is $\partial_i$-homogeneous for each $i$ with $1\leq i\leq s$. The statement follows by an induction argument.
\end{proof}

We will show that all R-gradings and all nice gradings are stable on $\Gr_{\ud}(M)$.

\begin{lem} Let $Q$ be a quiver, $M$ a $Q$-representation and $\ud$ a dimension vector. Then each R-grading $\partial$ is stable on $\Gr_{\ud}(M)$.
\end{lem}

For this lemma it is enough to show that $M_a$ and $\varphi_\partial(\lambda)$ commute for all $a\in Q_1$ and $\lambda\in\CC^\ast$.

 \begin{lem}\label{lem-auf-homos-testen} Let $\partial_1,\ldots,\partial_r$ and $\partial$ be gradings of $M$. Then $\partial$  is stable on the variety ${\Gr}_{\ud}(M)^{\partial_1,\ldots,\partial_r}$ for all $\ud\in\NN^{|Q_0|}$ if and only if for all $\lambda\in\CC^\ast$, $a\in Q_1$ and $\partial_1,\ldots,\partial_r$-homogeneous elements $u\in M$ we have \begin{align}\label{gl-auf-homos-testen}M_a\left(\varphi_{\partial}(\lambda)u\right)\in \varphi_{\partial}(\lambda)U_{\partial_1,\ldots,\partial_r}(u), \end{align} where  $U_{\partial_1,\ldots,\partial_r}(u)$ is the minimal subrepresentation of $M$ such that $u\in U_{\partial_1,\ldots,\partial_r}(u)$ and $U_{\partial_1,\ldots,\partial_r}(u)\in{\Gr}_{\ud}(M)^{\partial_1,\ldots,\partial_r}$ for some $\ud\in\NN^{|Q_0|}$.\end{lem}

If $U\in\Gr_{\ud}(M)^{\partial_1,\ldots,\partial_r}$ and $V\in\Gr_{\uc}(M)^{\partial_1,\ldots,\partial_r}$, then Lemma~\ref{lem-stab-homo} implies $U\cap V\in\Gr_{\udim(U\cap V)}(M)^{\partial_1,\ldots,\partial_r}$. So the submodule $U_{\partial_1,\ldots,\partial_r}(u)$ is well-defined and unique.

 \begin{proof} If $\partial$ is stable on ${\Gr}_{\ud}(M)^{\partial_1,\ldots,\partial_r}$ for all $\ud\in\NN^{|Q_0|}$, then $\varphi_{\partial}(\lambda)U_{\partial_1,\ldots,\partial_r}(u)$ is a subrepresentation of $M$ for all $\lambda\in\CC^\ast$ and $u\in M$.

 Let $U\in{\Gr}_{\ud}(M)^{\partial_1,\ldots,\partial_r}$. If Equation~\eqref{gl-auf-homos-testen} holds for all $\lambda\in\CC^\ast$, $a\in Q_1$ and $\partial_1,\ldots,\partial_r$-homogeneous $u\in M$, then  $M_a\left(\varphi_{\partial}(\lambda)U_{s(a)}\right)\subseteq\varphi_{\partial}(\lambda)U_{t(a)}$ for all $\lambda\in\CC^\ast$ and $a\in Q_1$, since $U$ is generated by $\partial_1,\ldots,\partial_r$-homogeneous elements. Thus $\varphi_{\partial}(\lambda)U\in{\Gr}_{\ud}(M)$.
\end{proof}

\begin{lem}\label{lem-nice-Grad-stabil} Let $Q$ be a quiver, $M$ a $Q$-representation and $\ud$ a dimension vector. Then every nice $\partial_1,\ldots,\partial_r$-grading $\partial$ is stable on $\Gr_{\ud}(M)^{\partial_1,\ldots,\partial_r}$.
\end{lem}
 \begin{proof} By Lemma~\ref{lem-auf-homos-testen}, it is enough to consider $\lambda\in\CC^\ast$, $a\in Q_1$ and a homogeneous $u\in M$. We can write $u=\sum_{k\in I} u_ke_k$ with $u_k\in\CC$, $M_a(e_k)=\sum_{j\in I}m_{jk}e_j$ with $m_{jk}\in\CC$ for all $k\in I$ and $M_a(u)=\sum_{\mathbf z\in\ZZ^r}m_{\mathbf z}$ with $(\partial_m(m_{\mathbf z}))_m=\mathbf z$. So $m_{\mathbf z}=\sum_{k,j\in I,(\partial_m(e_j))_m=\mathbf z}u_km_{jk}e_j $ and \begin{align*}
M_a\left(\varphi_{\partial}(\lambda)u\right)&
 ={\sum}_{k\in I}u_kM_a\left(\lambda^{{\partial}\left(e_k\right)}e_k\right)
 ={\sum}_{k,j\in I}u_k\lambda^{{\partial}\left(e_k\right)}m_{jk}e_{j}\displaybreak[0]\\&
 ={\sum}_{k,j\in I}\lambda^{\partial\left(e_k\right)-{\partial}\left(e_j\right)}u_km_{jk}\varphi_\partial(\lambda)e_{j}\displaybreak[0]\\&
 =\varphi_\partial(\lambda)\left({\sum}_{k,j\in I}\lambda^{\Delta((\partial_m(u))_m,(\partial_m(e_j))_m,a)}u_km_{jk}e_{j}\right)\displaybreak[0]\\& 
 =\varphi_\partial(\lambda)\left({\sum}_{\mathbf z\in\ZZ^r}\lambda^{\Delta((\partial_m(u))_m,\mathbf z,a)}m_{\mathbf z}\right)\in\varphi_\partial(\lambda)U_{\partial_1,\ldots,\partial_r}(u).\end{align*}
\end{proof}

\subsection{Ringel-Hall algebras}
 In the theory of Ringel-Hall algebras one has to compute the Euler characteristic of the following locally closed subsets of the projective variety ${\Gr}_{\ud}(M)$. Let $U$ and $M$ be $Q$-representations and $X\subseteq{\Gr}_{\ud}(M)$ a locally closed subset. Let \begin{align*}X^U:=\Big\{V\in X\Big|V\cong U,M/V\cong M/U\Big\}.\end{align*}

\begin{lem}  Let $Q$ be a quiver, $U$ and $M$ be $Q$-representations.  Then every R-grading $\partial$ is stable on ${\Gr}_{\ud}(M)^U$.
\end{lem}

\begin{proof} The linear map $\varphi_\partial(\lambda)\colon M\to M$ is an automorphism of $Q$-representations for all $\lambda\in\CC^\ast$.
 \end{proof}

 \begin{lem}\label{lem-stab-tree} Let $Q$ be a quiver, $M$ a $Q$-representation and $F_\ast(V)\subseteq M$ with $F\colon S\to Q$ a tree. Let $\partial$ be a nice grading on $\Gr_{\ud}(M)$ such that $M/F_\ast(V)$ is a tree module. Then $\partial$ is also stable on $\Gr_{\ud}(M)^{F_\ast(V)}$.
\end{lem}

\begin{proof}
 Let $a\in Q_1$, $\lambda\in\CC^\ast$ and $U\in\Gr_{\ud}(M)^{F_\ast(V)}$. Since $\partial$ is a nice grading we know $M_a\varphi_\partial(\lambda)=\lambda^{\partial(a)}\varphi_\partial(\lambda)M_a$ by the proof of Lemma~\ref{lem-nice-Grad-stabil}. Let $i\in S_0$ and $\rho_j$ the unique not necessarily oriented path in $S$ from $i$ to some $j\in S_0$. Then we can associate an integer $\partial(\rho_j)$ to each path $\rho_j$ such that $f_j\mapsto \lambda^{\partial(\rho_j)}f_j$ induces an isomorphism $U\to\varphi_\partial(\lambda)(U)$ of quiver representations. The same holds for the quotient.
\end{proof}

\begin{lem}\label{lem-stab-band}Let $Q$ and $S$ be quivers, $B\colon S\to Q$ a winding, $M$ a $Q$-representation and
$\partial$ a nice grading on $\Gr_{\ud}(M)$. Let \begin{align*}X=\Big\{U\in{\Gr}_{\ud}(M)\Big|\exists B_\ast(V)\text{ band module},U\cong B_\ast(V)\Big\},\end{align*} a locally closed subset of $\Gr_{\ud}(M)$. Then $\partial$ is also stable on $X$.
\end{lem}

\begin{proof}
 We can use the proof of Lemma~\ref{lem-stab-tree}. In this case the representations $U$ and $\varphi_\partial(\lambda)(U)$ are in general non-isomorphic. But they are both band modules for the same quiver $S$ and the same winding $B\colon S\to Q$.
\end{proof}

The next example shows that this lemma is not true if we restrict the action to one orbit of a band module.
\begin{ex}
 Let $F_\ast(V)$ be the tree module described by the following picture.
\begin{align*}F\colon\left(\vcenter{\begin{xy}\SelectTips{cm}{}\xymatrix@-1pc{
  1\ar[dr]_\alpha&&1'\ar[dl]^\beta\\
  &2
  }\end{xy}}\right)\to
\left(\vcenter{\begin{xy}\SelectTips{cm}{}\xymatrix@-1pc{
  1\ar@<-0.5ex>[d]_\alpha\ar@<0.5ex>[d]^\beta\\
  2
  }\end{xy}}\right)\end{align*} Let $U$ be the subrepresentation of $F_\ast(V)$ generated by $F_\ast(f_1+ f_{1'})$. Let $\lambda\in\CC^\ast$ with $\lambda\neq1$ and $\partial$ a nice grading of  $F_\ast(V)$  with $\partial(\alpha)=1$ and $\partial(\beta)=0$ (see Remark~\ref{rem-basis-FV}). Then $\varphi_\partial(\lambda)U$ is generated by $F_\ast(f_1+\lambda f_{1'})$, and $U$ and $\varphi_\partial(\lambda)U$ are non-isomorphic band modules.
\end{ex}

\section{Quiver Grassmannians}\label{thm-Quiv}

\subsection{Proof of Theorem~\ref{thm-Grad}}\label{sec-thm-Grad}
The following proposition is well known.

 \begin{prop}[Bialynicki-Birula \cite{BilaBiru}]\label{prop-Fixpkt} Let $\CC^\ast$ act on a locally closed subset $X$ of a projective variety $Y$. Then the subset of fixed points $X^{\CC^\ast}$ under this action is a locally closed subset of $Y$ and the Euler characteristic of $X$ equals the Euler characteristic of $X^{\CC^\ast}$. If the subset $X$ is non-empty and closed in $Y$, then $X^{\CC^\ast}$ is also non-empty and closed in $Y$.
 \end{prop}

\begin{proof} The subset of fixed points $X^{\CC^\ast}$ is closed in $X$. By \cite{Borel}, this is non-empty if $X$ is non-empty and closed in $Y$.

So we can decompose $X$ into the locally closed subset of fixed points $X^{\CC^\ast}$ and its complement $U=X-X^{\CC^\ast}$ in $X$. So $\chi(X)=\chi\left(X^{\CC^\ast}\right)+\chi(U)$. Since $U$ is the union of the non trivial orbits in $X$, the projection $U\to U/\CC^\ast$ is a algebraic morphism. Since $\chi(\CC^\ast)=0$ the Euler characteristic of $U$ is also zero.
\end{proof}

The action $\varphi_\partial$ of the algebraic group $\CC^\ast$ on the projective variety $X$ is well-defined. Thus Proposition~\ref{prop-Fixpkt} yields the equality of the Euler characteristic of $X$ and the Euler characteristic of the set of fixed points under this action. By Lemma~\ref{lem-stab-homo}, a subrepresentation $U$ of $M$ in $X$ is a fixed point of $\varphi_\partial$ if and only if $U$ has a basis of $\partial$-homogeneous elements. This proves Theorem~\ref{thm-Grad}.$\hfill \Box$

\begin{cor} Let $Q$ be a quiver, $M$ a $Q$-representation and $\partial_1,\ldots,\partial_r$ gradings of $M$ such that for all $1\leq i\leq r$ the grading $\partial_i$ is a stable grading on ${\Gr}_{\ud}(M)^{\partial_1,\ldots,\partial_{i-1}}$. Then the $\chi_{\ud}\left(M\right)=\chi\left({\Gr}_{\ud}(M)^{\partial_1,\ldots,\partial_r}\right)$.
\end{cor}This corollary follows directly from Theorem~\ref{thm-Grad}, since different $\CC^\ast$-actions commute.

\subsection{Proof of Proposition~\ref{prop-dir-sum}}\label{sec-proof-prop-dir-sum}

We choose any basis of $M$ and any basis of $N$. So the union induces a basis of $M\oplus N$. Using an R-grading $\partial$, we have to compute the Euler characteristic of the set of fixed points. This variety $\Gr_{\ud}(M\oplus N)^\partial$ can be decomposed into a union of locally closed sets $X_{\uc}$, where the subrepresentation of $M$ has the dimension vector $\uc$ and the subrepresentation of $N$ has the dimension vector $\ud-\uc$. Then $\chi_{\ud}(M)=\sum_{\uc}\chi(X_{\uc})=\chi_{\uc}(M)\chi_{\ud-\uc}(N)$.$\hfill \Box$

\section{Tree and band modules}\label{sec-tree-band}

 \subsection{Proof of Theorem~\ref{thm-tree-band}\eqref{part-tree-band-1}} Let $Q$ and $S$ be quivers, $F\colon S\to Q$ a winding and $V$ a $S$-representation.

It is enough to consider the case $\dim_{\CC}(V_i)=1$ for all $i\in S_0$. By Remark~\ref{rem-basis-FV}, the set $E=\{F_\ast(f_i)|i\in S_0\}$ is a basis of $F_\ast(V)$. We write $\partial(i)$ instead of $\partial(F_\ast(f_i))$ for all $i\in S_0$.

 To prove Theorem~\ref{thm-tree-band}\eqref{part-tree-band-1} we can use inductively Theorem~\ref{thm-Grad} and Lemma~\ref{lem-auf-homos-testen} and~\ref{lem-nice-Grad-stabil}. Let $\partial$ be a nice grading of $F_\ast(V)$. Define a new quiver $Q'$ by \begin{align*}Q'_0=&\{(F_0(i),\partial(i))|i\in S_0\}\\ Q'_1=&\{(\partial(s(a)),\partial(t(a)),F_1(a))|a\in S_1\}\\ s'(\partial(s(a)),\partial(t(a)),F_1(a))=&(s(F_1(a)),\partial(s(a)))\\ t'(\partial(s(a)),\partial(t(a)),F_1(a))=&(t(F_1(a)),\partial(t(a)))\text{ for all $a\in Q_1$.}\end{align*} Define windings $F'\colon S\to Q'$ by \begin{align*} i\mapsto& (F_0(i),\partial(i))\\a\mapsto& (\partial(s(a)),\partial(t(a)),F_1(a))\end{align*} and $G\colon Q'\to Q$ by \begin{align*}(F_0(i),\partial(i))\mapsto&F_0(i)\\(\partial(s(a)),\partial(t(a)),F_1(a))\mapsto&F_1(a).\end{align*} Then $F=GF'$ and by Theorem~\ref{thm-Grad} $\chi_{\ud}(F_\ast(V))=\sum_{\ut\in G^{-1}(\ud)}\chi_{\ut}\left(F'_\ast(V)\right)$.

\begin{ex} We have a look at Example~\ref{ex-rem-basis-FV}. Let $Q'$ and $F'$ be described by the following picture.
\begin{align*}S=\left(\vcenter{\begin{xy}\SelectTips{cm}{}\xymatrix@-0.8pc{
  1\ar[d]_\alpha&&1'\ar[d]^{\beta'}\\
  2\ar[dr]_\beta&&2'\ar[dl]^{\alpha'}\\
  &3
  }\end{xy}}\right)\stackrel{F'}\longrightarrow Q'=\left(\vcenter{\begin{xy}\SelectTips{cm}{}\xymatrix@-0.8pc{
  1\ar@<-0.5ex>[d]_\alpha\ar@<0.5ex>[d]^{\beta'}\\
  2\ar@<-0.5ex>[d]_\beta\ar@<0.5ex>[d]^{\alpha'}\\
  3
  }\end{xy}}\right)\stackrel{G}\longrightarrow Q=\Big(\vcenter{\begin{xy}\SelectTips{cm}{}\xymatrix@+0pc{
  \circ\ar@(ul,dl)[]_\alpha\ar@(ur,dr)[]^\beta}\end{xy}}\Big) \end{align*} Using the nice grading $\partial_1$, it is enough to observe $F'_\ast(V)$ and $\chi_{\ut}\left(F'_\ast(V)\right)$ to compute $\chi_{\ud}(F_\ast(V))$. So the nice $\partial_1$-grading $\partial_2$ induces a nice grading of $F'_\ast(V)$.
\end{ex}

\begin{lem}\label{lem-tree-gl-tree-band} Equation~\eqref{gl-tree-band} holds for each tree module $F_\ast(V)$.
\end{lem}

\begin{proof} By the previous observation, it is enough to treat the cases when $F_0\colon S_0\to Q_0$ is not injective. If $i,j\in S_0$ exist with $F_0(i)=F_0(j)$ and $i\neq j$, we construct a nice grading $\partial$ of $F_\ast(V)$ such that $\partial(i)\neq\partial(j)$. 

 Let $S'$ be a minimal connected subquiver of $S$ such that there exist $i,j\in S'_0$  with $F_0(i)=F_0(j)$ and $i\neq j$. Then $S'$ is of type $A_l$. Let $F'\colon S'\to Q$ be the winding induced by $F$ and $V'$ the $S'$-representation induced by $V$. Since $S$ is a tree, every nice grading of $F'_\ast(V')$ can be extended to a nice grading of $F_\ast(V)$.

 So without loss of generality let $S'$ be equal to $S$. So $S_0=\{1,\ldots,l\}$ and  $S_1=\{s_1,s_2,\ldots,s_{l-1}\}$ as in Section~\ref{sec-mor-rep} and $F_0(1)=F_0(l)$ and $1<l$. So $\partial\colon S_0\to \ZZ, i\mapsto\delta_{i1}$ defines a grading of $F_\ast(V)$ with $\partial(1)=1\neq0=\partial(l)$. Since $F_0(2)\neq F_0(l-1)$, we have $F_1(s_1)^{-\varepsilon_1}\neq F_1(s_{l-1})^{\varepsilon_{l-1}}$ and so for all $1<k<l$ the equation $F_1(s_1)\neq F_1(s_k)$ holds by the minimality of $S$. Therefore $\partial$ is a nice grading.
\end{proof}

\begin{lem}\label{lem-band-gl-tree-band} Equation~\eqref{gl-tree-band} holds for each band module $F_\ast(V)$.
\end{lem}

\begin{proof}  Let $i,j\in S_0$ with $F_0(i)=F_0(j)$, $i<j$ and $j-i$ minimal (i.e.\ $F_0(k)\neq F_0(m)$ for all $k,m\in S_0$ with $i\leq k<m\leq j$ and $(i,j)\neq(k,m)$). If no such tuple $(i,j)\in S^2_0$ exists, we are done. By the previous observations, it is again enough to construct a nice grading $\partial$ of $F_\ast(V)$ such that $\partial(i)\neq\partial(j)$.

For each $a\in Q_1$ let $\rho(a):=\sum_{i=1}^l\varepsilon_i\delta_{a,F_1(s_i)}$ and $\partial^{(a)}\colon Q_1\to\ZZ,b\mapsto\delta_{ab}$ a map.
\begin{itemize}
 \item If $a\in Q_1$ with $\rho(a)=0$, then $\partial^{(a)}$ induces a nice grading $\partial^{(a)}$ of $F_\ast(V)$ such that \begin{align*}\partial^{(a)}(i)-\partial^{(a)}(j)={\sum}_{k=i}^{j-1}\varepsilon_k\delta_{a,F_1(s_k)}.\end{align*}

 \item If $a,b\in Q_1$, then $\partial^{(a,b)}:=\rho(a)\partial^{(b)}-\rho(b)\partial^{(a)}$ induces similarly a nice grading $\partial^{(a,b)}$ of $F_\ast(V)$.
\end{itemize}
If $\rho(F_1(s_i))=0$, then $\partial^{(F_1(s_i))}(i)-\partial^{(F_1(s_i))}(j)=\varepsilon_i$ since $j-i$ is minimal. Thus $F_1(s_i)\neq F_1(s_k)$ for all $k\in S_0$ with $i<k<j$.

If $\rho(F_1(s_i))\neq0$, we should have a look at $\partial^{(F_1(s_i),F_1(s_k))}$ for all $k\in S_0$. If $\partial^{(F_1(s_i),F_1(s_k))}(i)-\partial^{(F_1(s_i),F_1(s_k))}(j)\neq0$ for some $k\in S_0$, we are done. So let us assume $\partial^{(F_1(s_i),F_1(s_k))}(i)-\partial^{(F_1(s_i),F_1(s_k))}(j)=0$ for all $k\in S_0$ and for all tuples $(i,j)\in S^2_0$ with $0<j-i$ minimal. If $F_1(s_i)\neq F_1(s_k)$, then \begin{align*}0=&\partial^{(F_1(s_i),F_1(s_k))}(i)-\partial^{(F_1(s_i),F_1(s_k))}(j)\\
 =&\rho(F_1(s_i))\left({\sum}_{m=i+1}^{j-1}\varepsilon_m\delta_{F_1(s_k),F_1(s_m)}\right)-\rho(F_1(s_k))\varepsilon_i\\
 =&\rho(F_1(s_i))\varepsilon_{k'}-\rho(F_1(s_k))\varepsilon_i\end{align*} for some $k'\in S_0$ with $i<k'<j$ and $F_1(s_k)=F_1(s_{k'})$. So $\varepsilon_k\rho(F_1(s_k))=\varepsilon_m\rho(F_1(s_m))$ for all $k,m\in S_0$. In other words, $\rho(F_1(s_k))\neq0$ for all $k\in S_1$ and $\varepsilon_k=\varepsilon_m$ for all $k,m\in S_0$ with $F_1(s_k)=F_1(s_m)$. So some $r\in\NN_{>0}$ exists such that $F_1(s_k)=F_1(s_{k+r})$ for all $k\in S_0$. By Example~\ref{ex-band-per}, the representation $F_\ast(V)$ is decomposable if $r<l$. This is a contradiction.
\end{proof}

\subsection{Proof of Theorem~\ref{thm-tree-band}\eqref{part-tree-band-2}}\label{sec-tree-band-2}

 Let $S$ be a quiver of type $\widetilde A_{l-1}$ and $\{i_1,\ldots,i_r\}$ be the sources and $\{i'_1,\ldots,i'_r\}$ be the sinks of $S$. It is visualized in Figure~\ref{fig-tAl}. For all $i,j\in S_0$ with $i<j$ let $S_{ij}$ be the full subquiver of $S$ with $(S_{ij})_0=\{i,i+1,\ldots,j\}$.

 \begin{lem}\label{lem-band} Let $S$ be a quiver as above and $V=(V_i,V_{s_i})_{i\in S_0}$ a band module. Let $\ut=(t_1,\ldots,t_l)$ be a dimension vector of $S$ and $n:=\dim_{\CC}(V_i)$ for some $i\in S_0$. Then \begin{align}\label{gl-lem-band}\chi_{\ut}\left(V\right)=\sum_{k\in\ZZ}\begin{pmatrix}t_{i_1}\\k\end{pmatrix}\begin{pmatrix}n-t_{i_1}\\k\end{pmatrix} X^{(1,1)}_{t_{i_1}-k,k,k,n-t_{i_1}-k}(\ut')\end{align} with \begin{align*}\ut'=\begin{cases}
 (0,t_{i_1+1}-t_{i_1},\ldots,t_{i'_1-1}-t_{i_1},\\\qquad t_{i'_1}-t_{i_1}-k,t_{i'_1+1}-t_{i_1},\ldots,t_{i_1-1}-t_{i_1})&\text{ if }r=1,\\
 (0,t_{i_1+1}-t_{i_1},\ldots,t_{i'_1}-t_{i_1},t_{i'_1+1},\ldots,t_{i'_r-1},\\\qquad t_{i'_r}-t_{i_1},\ldots,t_{i_1-1}-t_{i_1})&\text{ if }r>1.\end{cases}\end{align*} For all $s,t\in S_0$ and $\alpha,\beta,\gamma,\delta\in\NN$ we define $X^{(s,t)}_{\alpha,\beta,\gamma,\delta}(\ut)$ to be 
 \begin{align*}\chi_{\ut}\left(M(S_{i'_s+1,i'_r-1})^\alpha\oplus M(S_{i'_s+1,i_1-1})^\beta\oplus M(S_{i_t+1,i'_r-1})^\gamma\oplus M(S_{i_t+1,i_1-1})^\delta\right),
 \end{align*} where $M(S_{ij})$ is an indecomposable $S_{ij}$-representation with dimension $j-i+1$ for all $i,j\in S_0$ with $i<j$. 
\end{lem}
(We use here the convention $\begin{pmatrix}r\\s\end{pmatrix}=0$ for all $r,s\in\ZZ$ if $s<0$ or $s>r$.) We visualize the representations $M(S_{i'_s+1,i'_r-1})$, $M(S_{i'_s+1,i_1-1})$, $M(S_{i_t+1,i'_r-1})$ and $M(S_{i_t+1,i_1-1})$ in Figure~\ref{fig-lem-band}.
\begin{figure}[ht]
$$\vcenter{\begin{xy}\SelectTips{cm}{}\xymatrix@-1.9pc{
&i_r\ar[ddl]\ar[dddr]&&&&i_{s+1}\ar[dddl]\ar[ddr]&&  &&&&i_r\ar[dddl]\ar[dddr]&&&&i_{s+1}\ar[dddl]\ar[ddr]\\
&&&&&&&                                              &&i_1-1\ar[ddr]\\
i'_r-1&&&&&&i'_s+1&                                  &&&&&&&&&i'_s+1\\
&&&\ldots&&&  &                                &&&i'_r&&&\ldots&&&\\
&i_r\ar[ddl]\ar[dddr]&&&&i_{t+1}\ar[dddl]\ar[dddr]&  &&&&&i_r\ar[dddl]\ar[dddr]&&&&i_{t+1}\ar[dddl]\ar[dddr]\\
&&&&&&&i_t+1\ar[ddl]                                 &&i_1-1\ar[ddr]&&&&&&&&i_t+1\ar[ddl]\\
i'_r-1&&&&&&&                                        \\
&&&\ldots&&&i'_t&                              &&&i'_r&&&\ldots&&&i'_t
} \end{xy}}$$
\caption{Modules occurring in the definition of $X^{(s,t)}_{\alpha,\beta,\gamma,\delta}(\ut)$
}\label{fig-lem-band}
\end{figure}
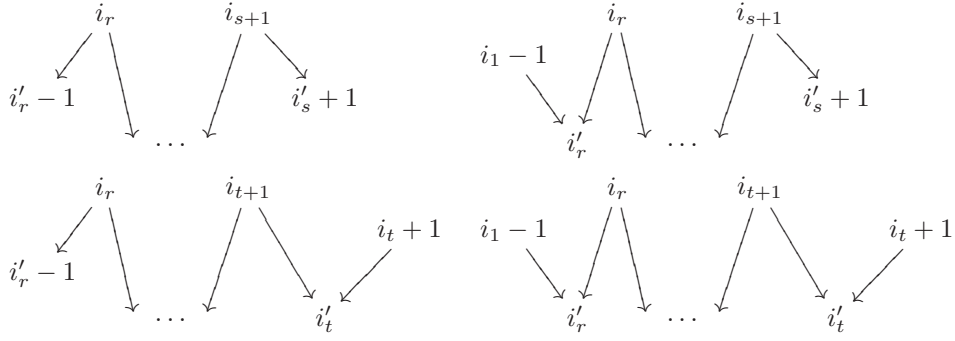

\begin{proof}[Proof of Lemma \ref{lem-band}] Let $\{e_{ik}|i\in S_0,1\leq k\leq n\}$ be a basis of $V$ such that the following holds.\begin{enumerate}
 \item  For all $1\leq m\leq n$, the vector space $V^{(m)}:=\langle e_{i,k}|i\in S_0,1\leq k\leq m\rangle$ is a subrepresentation of $V$ and a band module.

 \item There exists a nilpotent endomorphism $\psi$ of $V$ such that $\psi(e_{i1})=0$ and $\psi(e_{ik})=e_{i,k-1}$ for all $1<k\leq n$ and all $i\in S_0$.
 \end{enumerate}
Let $U=(U_i,V_{s_i}|_{U_i})_{i\in S_0}\in{\Gr}_{\ut}(V)$. Using the Gauß algorithm, a unique tuple \begin{align}\label{gl-band-1}\uj(U):=(1\leq j_1<j_2<\ldots<j_{t_{i_1}} \leq n)\end{align} and unique $\lambda_{kj}(U)\in\CC$ exist such that the vector space $U_{i_1}$ is generated by \begin{align*}e_{i_1j_m}+{\sum}_{j=1,j\neq j_k\forall k}^{j_m-1}\lambda_{mj}(U)e_{i_1j}\end{align*} with $1\leq m\leq t_{i_1}$. The variety ${\Gr}_{\ut}(V)$ can be decomposed into a disjoint union of locally closed subsets \begin{align*}{\Gr}_{\ut}(V)_{\uj}:=\Big\{U\subseteq V\Big|\udim(U)=\ut,\uj(U)=\uj\Big\},\end{align*} where $\uj\in\NN^{t_{i_1}}$. For each such tuple $\uj$ let \begin{align*}{\Gr}_{\ut}(V)_{\uj}^0:=\Big\{U\subseteq V\Big|\udim(U)=\ut,\uj(U)=\uj,\lambda_{1j}(U)=0\forall j\Big\}.\end{align*} These are locally closed subsets of ${\Gr}_{\ut}(V)$. The projection $\pi\colon{\Gr}_{\ut}(V)_{\uj}\to {\Gr}_{\ut}(V)_{\uj}^0$ with \begin{align}\label{gl-band-2}U\mapsto \prod_{j=1}^{j_1-1}\left(1+\lambda_{1j}(U)\psi^{j_1-j}\right)^{-1}(U)\end{align} is an algebraic morphism with affine fibers. For $U\in{\Gr}_{\ut}(V)_{\uj}^0$ let $U_{\uj}$ be the subrepresentation of $V$ generated by $e_{i_1j_1}$. Let $V_{\uj}$ be the subrepresentation of $V$ with vector space basis  \begin{align*}\{e_{i_1k}|j_1\leq k\leq n\}\cup\{e_{ik}|i_1\neq i\in S_0,1\leq k\leq n\}.\end{align*} Then $U_{\uj}\subseteq U\subseteq V_{\uj}\subseteq V$ and
\begin{align*} \udim\left(U_{\uj}\right)_i=\left\{\begin{array}{cl}
0&  \text{ if } i'_1<i<i'_r\\
2&  \text{ if } r=1, i_1'=i \text{ and } j_1>1\\
1&  \text{ otherwise.}\end{array}\right.\end{align*}
So if $j_1=1$ we get 
\begin{align}\label{gl-band-3} V_{\uj}/U_{\uj}\cong& V^{(n-1)}\oplus M\left(S_{i'_1+1,i_r'-1}\right),\end{align} and if $j_1>1$ we get
\begin{align}\label{gl-band-4} V_{\uj}/U_{\uj}\cong& V^{(n-j_1)}\oplus M\left(S_{i'_1+1,i_1-1}\right)\oplus M\left(S_{i_1+1,i_r'-1}\right)\oplus M\left(S_{i_1+1,i_1-1}\right)^{j_1-2}.\end{align}
 Thus \begin{align}\label{gl-proof-lem-band-3}
  \chi_{\ut}(V)={\sum}_{\uj\in\NN^{t_{i_1}}}\chi\left({\Gr}_{\ut}(V)_{\uj}^0\right)={\sum}_{\uj\in\NN^{t_{i_1}}}\chi_{\ut-\udim(U_{\uj})}\left(V_{\uj}/U_{\uj}\right).
\end{align}
Let $n_{\uj}:=\left|\left\{1\leq i\leq n\middle|i\neq j_m\forall m,\exists m:i+1=j_m\right\}\right|$. A simple calculation shows \begin{align*}\left|\left\{\uj\middle|n_{\uj}=k\right\}\right|=\begin{pmatrix}t_{i_1}\\k\end{pmatrix}\begin{pmatrix}n-t_{i_1}\\k\end{pmatrix}. \end{align*} We do an induction over $t_{i_1}$. Then Equation~\eqref{gl-band-4} occurs $n_{\uj}$-times, Equation~\eqref{gl-band-3} occurs $(t_{i_1}-n_{\uj})$-times and so Equation~\eqref{gl-lem-band} holds in general by an inductive version of Equation~\eqref{gl-proof-lem-band-3}.
\end{proof}

The rest of the proof of Theorem~\ref{thm-tree-band}\eqref{part-tree-band-2} is done in the next two combinatorial lemmas.

\begin{lem}\label{lem-binomial}
 Let $a,b,c,d,e,f\in\NN$. Then  \begin{align*} \begin{pmatrix}a\\b\end{pmatrix}\begin{pmatrix}b\\c\end{pmatrix}=\begin{pmatrix}a-c\\b-c\end{pmatrix} \begin{pmatrix}a\\c\end{pmatrix}\end{align*} and
\begin{align*}\sum_{m\in\ZZ}\begin{pmatrix}a\\b-m\end{pmatrix}\begin{pmatrix}b-m\\c\end{pmatrix}\begin{pmatrix}d\\e+m\end{pmatrix}\begin{pmatrix}e+m\\f\end{pmatrix}=\begin{pmatrix}a\\c\end{pmatrix}\begin{pmatrix}d\\f\end{pmatrix}\begin{pmatrix}a+d-c-f\\a+d-b-e\end{pmatrix}.\end{align*}
\end{lem}

\begin{proof}
 The first equation can be shown using the definition. The second equation is a consequence of the first one.
\end{proof}

 \begin{lem}\label{lem-band2} Let $S$, $V$, $\ut$, $n$ as above and $1\leq m\leq r$. Then \begin{align}\label{gl-lem-band2}\chi_{\ut}\left(V\right)=&\Lambda_m\Gamma_{i_1i_m}
 \sum_{k\in\ZZ}\begin{pmatrix}t_{i_1}\\ t_{i_m}-k\end{pmatrix}\begin{pmatrix}n-t_{i_1}\\k\end{pmatrix} X^{(m,m)}_{t_{i_m}-k,k,t_{i_1}-t_{i_m}+k,n-t_{i_1}-k}(\ut')\end{align} with  \begin{align*}\Lambda_m:=\prod_{k=1}^{m-1}\frac{(n-t_{i_{k+1}})!}{t_{i_k}!}\frac{t_{i'_k}!}{(n-t_{i'_k})!},\qquad\Gamma_{ij}:=\prod_{k=i}^{j-1}\frac1{(\varepsilon_k(t_k-t_{k+1}))!}\end{align*} and \begin{align*}\ut'=\begin{cases}
 (0,\ldots,0,t_{i_r+1}-t_{i_r},\ldots,t_{i'_r-1}-t_{i_r},\\\qquad t_{i'_r}-t_{i_1}-k,t_{i'_r+1}-t_{i_1},\ldots,t_{i_1-1}-t_{i_1}) &\text{ if }m=r,\\
 (0,\ldots,0,t_{i_m+1}-t_{i_m},\ldots,t_{i'_m}-t_{i_m},t_{i'_m+1},\ldots,t_{i'_r-1},\\\qquad t_{i'_r}-t_{i_1},\ldots,t_{i_1-1}-t_{i_1})&\text{ if }m<r.\end{cases}\end{align*}

\end{lem}

\begin{proof} For $m=1$ this is the statement of Lemma~\ref{lem-band}. We prove the lemma by induction. Let $1<m\leq r$. Then $\chi_{\ut}\left(V\right)=$
\begin{align*}&\Lambda_{m-1}\Gamma_{i_1i_{m-1}}\sum_k\begin{pmatrix}t_{i_1}\\ t_{i_{m-1}}-k\end{pmatrix}\begin{pmatrix}n-t_{i_1}\\k\end{pmatrix} X^{(m-1,m-1)}_{t_{i_{m-1}}-k,k,t_{i_1}-t_{i_{m-1}}+k,n-t_{i_1}-k}(\ut')\displaybreak[0]\\
 =&\Lambda_{m-1}\Gamma_{i_1i_{m-1}}\sum_k\begin{pmatrix}t_{i_1}\\ t_{i_{m-1}}-k\end{pmatrix}\begin{pmatrix}n-t_{i_1}\\k\end{pmatrix}\sum_p\begin{pmatrix}t_{i_1}-t_{i_{m-1}}+k\\t_{i'_{m-1}}-t_{i_{m-1}}-p\end{pmatrix}\begin{pmatrix}n-t_{i_1}-k\\p\end{pmatrix}\\
 &(t_{i'_{m-1}}-t_{i_{m-1}})!\Gamma_{i_{m-1}i'_{m-1}}X^{(m-1,m)}_{t_{i'_{m-1}}-k-p,k+p,t_{i_1}-t_{i'_{m-1}}+k+p,n-t_{i_1}-k-p}(\ut'')\displaybreak[0]\\
 =&\Lambda_{m-1}\Gamma_{i_1i'_{m-1}}\sum_p\left(\sum_k\begin{pmatrix}t_{i_1}\\ t_{i_{m-1}}-k\end{pmatrix}\begin{pmatrix}n-t_{i_1}\\k\end{pmatrix}\begin{pmatrix}t_{i_1}-t_{i_{m-1}}+k\\t_{i_1}-t_{i'_{m-1}}+p\end{pmatrix}\begin{pmatrix}n-t_{i_1}-k\\n-t_{i_1}-p\end{pmatrix}\right)\\
 &(t_{i'_{m-1}}-t_{i_{m-1}})!X^{(m-1,m)}_{t_{i'_{m-1}}-p,p,t_{i_1}-t_{i'_{m-1}}+p,n-t_{i_1}-p}(\ut'')
\end{align*}
with \begin{align*}\ut''=(0,\ldots,0,t_{i'_{m-1}+1},\ldots,t_{i'_r+1}-t_{i_1},\ldots,t_{i_1-1}-t_{i_1}).\end{align*} Lemma~\ref{lem-binomial} yields $\chi_{\ut}\left(V\right)=$
\begin{align*}
  &\Lambda_{m-1}\Gamma_{i_1i'_{m-1}}\sum_p\begin{pmatrix}n-t_{i_1}\\p\end{pmatrix}\begin{pmatrix}t_{i_1}\\t_{i'_{m-1}}-p\end{pmatrix}\begin{pmatrix}t_{i'_{m-1}}\\t_{i_{m-1}}\end{pmatrix}\\
  &(t_{i'_{m-1}}-t_{i_{m-1}})!X^{(m-1,m)}_{t_{i'_{m-1}}-p,p,t_{i_1}-t_{i'_{m-1}}+p,n-t_{i_1}-p}(\ut'')\displaybreak[0]\\
 =&\Lambda_{m-1}\frac{t_{i'_{m-1}}!}{t_{i_{m-1}}!}\Gamma_{i_1i'_{m-1}}\sum_p\begin{pmatrix}n-t_{i_1}\\p\end{pmatrix}\begin{pmatrix}t_{i_1}\\t_{i'_{m-1}}-p\end{pmatrix}\sum_k\begin{pmatrix}t_{i'_{m-1}}-p\\t_{i_m}-k\end{pmatrix}\begin{pmatrix}p\\k\end{pmatrix}\\
   &(t_{i'_{m-1}}-t_{i_m})!\Gamma_{i'_{m-1}i_m}X^{(m,m)}_{t_{i_m}-k,k,t_{i_1}-t_{i_m}+k,n-t_{i_1}-k}(\ut')\displaybreak[0]\\
 =&\Lambda_{m-1}\frac{t_{i'_{m-1}}!}{t_{i_{m-1}}!}\Gamma_{i_1i_m}\sum_k\left(\sum_p\begin{pmatrix}n-t_{i_1}\\p\end{pmatrix}\begin{pmatrix}t_{i_1}\\t_{i'_{m-1}}-p\end{pmatrix}\begin{pmatrix}t_{i'_{m-1}}-p\\t_{i_m}-k\end{pmatrix}\begin{pmatrix}p\\k\end{pmatrix}\right)\\
  &(t_{i'_{m-1}}-t_{i_m})!X^{(m,m)}_{t_{i_m}-k,k,t_{i_1}-t_{i_m}+k,n-t_{i_1}-k}(\ut').
\end{align*}
Using Lemma~\ref{lem-binomial} again, we get $\chi_{\ut}\left(V\right)=$
\begin{align*}
 &\Lambda_{m-1}\frac{t_{i'_{m-1}}!}{t_{i_{m-1}}!}\Gamma_{i_1i_m}\sum_k\begin{pmatrix}t_{i_1}\\t_{i_m}-k\end{pmatrix}\begin{pmatrix}n-t_{i_1}\\k\end{pmatrix}\begin{pmatrix}n-t_{i_m}\\n-t_{i'_{m-1}}\end{pmatrix}\\
  &(t_{i'_{m-1}}-t_{i_m})!X^{(m,m)}_{t_{i_m}-k,k,t_{i_1}-t_{i_m}+k,n-t_{i_1}-k}(\ut')\displaybreak[0]\\
 =
 &\Lambda_m\Gamma_{i_1i_m}\sum_k\begin{pmatrix}t_{i_1}\\t_{i_m}-k\end{pmatrix}\begin{pmatrix}n-t_{i_1}\\k\end{pmatrix}X^{(m,m)}_{t_{i_m}-k,k,t_{i_1}-t_{i_m}+k,n-t_{i_1}-k}(\ut').
\end{align*}
\end{proof}

 \begin{cor}\label{cor-band} Let $S$, $V$, $\ut$ and $n$ as above. Then Equation~\eqref{gl-prop-band} holds.
\end{cor}

\begin{proof} We have to show $\chi_{\ut}\left(V\right)=\Lambda_{r+1}\Gamma_{1,l+1}$ with $t_{i_{r+1}}=t_{i_1}$. Lemma~\ref{lem-band2} implies \begin{align*}\chi_{\ut}\left(V\right)
 =&\Lambda_r\Gamma_{i_1i_r}\sum_k\begin{pmatrix}t_{i_1}\\ t_{i_r}-k\end{pmatrix}\begin{pmatrix}n-t_{i_1}\\k\end{pmatrix} X^{(r,r)}_{t_{i_r}-k,k,t_{i_1}-t_{i_r}+k,n-t_{i_1}-k}(\ut')\end{align*} with 
 \begin{align*}\ut'=(0,\ldots,0,t_{i_r+1}-t_{i_r},\ldots,t_{i'_r-1}-t_{i_r},t_{i'_r}-t_{i_1}-k,t_{i'_r+1}-t_{i_1},\ldots,t_{i_1-1}-t_{i_1}).\end{align*} So we have  $\chi_{\ut}\left(V\right)=$
\begin{align*}
 &\Lambda_r\Gamma_{i_1i_r}\sum_k\begin{pmatrix}t_{i_1}\\ t_{i_r}-k\end{pmatrix}\begin{pmatrix}n-t_{i_1}\\k\end{pmatrix}\begin{pmatrix}n-t_{i_1}-k\\t_{i'_r}-t_{i_1}-k\end{pmatrix} X^{(r,r)}_{t_{i_r}-k,t_{i_r'}-t_{i_1},t_{i_r'}-t_{i_r},n-t_{i'_r}}(\ut''')\displaybreak[0]\\
 =&\Lambda_r\Gamma_{i_1i_r}\sum_k\begin{pmatrix}t_{i_1}\\ t_{i_r}-k\end{pmatrix}\begin{pmatrix}n-t_{i_1}\\k\end{pmatrix}\begin{pmatrix}n-t_{i_1}-k\\t_{i'_r}-t_{i_1}-k\end{pmatrix} (t_{i'_r}-t_{i_1})!\Gamma_{i'_ri_1}(t_{i'_r}-t_{i_r})!\Gamma_{i_ri'_r}
\end{align*} with 
 \begin{align*}\ut'''=(0,\ldots,0,t_{i_r+1}-t_{i_r},\ldots,t_{i'_r-1}-t_{i_r},0,t_{i'_r+1}-t_{i_1},\ldots,t_{i_1-1}-t_{i_1}).\end{align*}
Using Lemma~\ref{lem-binomial} again, we obtain
\begin{align*}
 \chi_{\ut}\left(V\right)=&\Lambda_r\Gamma_{1l}\begin{pmatrix}n-t_{i_1}\\ n-t_{i'_r}\end{pmatrix}\begin{pmatrix}t_{i'_r}\\t_{i_r}\end{pmatrix} (t_{i'_r}-t_{i_1})!(t_{i'_r}-t_{i_r})!=
 \Lambda_{r+1}\Gamma_{1,l+1}.
\end{align*} 
\end{proof}

\section{Quiver flag varieties}\label{sec-flag}
In this section we explain and justify Example~\ref{ex-fla-Kron}. Let $B\colon Q\to Q$ be the identity winding. For each $\mu\in\CC$ there is an automorphism of the algebra $\CC Q$ such that $B_\ast(\lambda,n)$ is mapped to $B_\ast(\lambda-\mu,n)$. This is not necessarily a band module. So we can assume without loss of generality that $M$ is a string module. Let \begin{align*} T^{(n)}=&\left(\vcenter{\begin{xy}\SelectTips{cm}{}\xymatrix@-1pc{
  1^{(1)}\ar[dr]&&1^{(2)}\ar[dl]\ar[dr]&&&&1^{(n)}\ar[dl]\ar[dr]\\
  &2^{(1)}&&2^{(2)}&\ldots&2^{(n-1)}&&2^{(n)}}\end{xy}}\right),\\
U^{(n)}=&\left(\vcenter{\begin{xy}\SelectTips{cm}{}\xymatrix@-1pc{
  1^{(1)}\ar[dr]&&1^{(2)}\ar[dl]\ar[dr]&&&&1^{(n)}\ar[dl]\\
  &2^{(1)}&&2^{(2)}&\ldots&2^{(n-1)}}\end{xy}}\right).
\end{align*} For any subquiver $V$ of $T^{(n)}$ let \begin{align*}\udim(V)=\Big(\left|\left\{i\middle|1^{(i)}\in V_0\right\}\right|,\left|\left\{i\middle|2^{(i)}\in V_0\right\}\right|\Big).\end{align*}
 For any dimension vectors $\uc$ and $\ud$ of $Q$ we set \begin{align*}X_{\uc}(V)=\Big\{&R\subseteq V\Big|R\text{ is a successor closed subquiver of }V,\udim(R)=\uc\Big\},\\
X_{\uc,\ud}(V)=\Big\{&R\subseteq S\subseteq V\Big|R\in X_{\uc}(S),S\in X_{\ud}(V)\Big\}. \end{align*} Using Corollary~\ref{cor-tree-band}, it is enough to show the following equality.\begin{align*}&\left|X_{(1,2),(2,3)}\left(T^{(n)}\right)\right|\displaybreak[0]\\ =&\sum_{i=1}^n\left|\left\{\left(R\subseteq S\right)\in X_{(1,2),(2,3)}\left(T^{(n)}\right)\middle|1^{(i)}\in R_0\right\}\right|\displaybreak[0]\\
 =&\left|X_{(0,1),(1,2)}\left(T^{(n-1)}\right)\right|+\sum_{i=2}^n\left| X_{(1,1)}\left(U^{(i-1)}\dot\cup T^{(n-i)}\right)\right|\displaybreak[0]\\
 =&\begin{pmatrix}2\\1\end{pmatrix}\left|X_{(1,2)}\left(T^{(n-1)}\right)\right|+\left(\left| X_{(0,1)}\left(T^{(n-2)}\right)\right|+\left| X_{(1,1)}\left(T^{(n-2)}\right)\right|\right)\\+&\sum_{i=3}^{n-1}\left(\left| X_{(1,1)}\left(U^{(i-1)}\right)\right|+\left| X_{(1,1)}\left(T^{(n-i)}\right)\right|\right)+\left| X_{(1,1)}\left(U^{(n-1)}\right)\right|\displaybreak[0]
\\
 =&2\left(\begin{pmatrix}n-2\\1\end{pmatrix}+\begin{pmatrix}n-2\\1\end{pmatrix}\right)+\left(\begin{pmatrix}n-2\\1\end{pmatrix}+1\right)+(n-3)\left(2+1\right)+2\displaybreak[0]\\ =&8(n-2)
\end{align*}

\section{Ringel-Hall algebras} \label{sec-RH-alg} 

\subsection{Reduction to indecomposable representations}\label{sec-red-ind}
In this section we look at the products of functions of the form $\Eins_{\uF,\uB,\un}$ in $\mathcal H(A)$, and we show that these are determined by the images of indecomposable $A$-modules.

\begin{lem}\label{lem-ind-enough}
Let $A$ be a finite-dimensional $\CC$-algebra, $f,g\in C(A)$ and $M$ and $N$ be $A$-modules. Then \begin{align*} (f\ast g)(M\oplus N)={\sum}_{i,j}\left(f^{(1)}_i\ast g^{(1)}_j\right)(M)\left(f^{(2)}_i\ast g^{(2)}_j\right)(N),\end{align*} where $\Delta(f)=\sum_if^{(1)}_i\otimes f^{(2)}_i$ and $\Delta(g)=\sum_jg^{(1)}_j\otimes g^{(2)}_j$.
\end{lem}

\begin{proof}
 Since $C(A)$ is a bialgebra the comultiplication $\Delta$ is an algebra homomorphism.
\end{proof}

\begin{ex}\label{ex-ind-enough}
 Let $\uF$ be a tuple of trees, $\uB$ a tuple of bands and $\un$ a tuple of positive integers. Then
  \begin{align*}\Delta(\Eins_{\uF,\uB,\un})=\sum_{\uF^{(1)}\dot\cup\uF^{(2)}=\uF,\uB^{(1)}\dot\cup\uB^{(2)}=\uB,\un^{(1)}\dot\cup\un^{(2)}=\un}\Eins_{\uF^{(1)},\uB^{(1)},\un^{(1)}}\otimes \Eins_{\uF^{(2)},\uB^{(2)},\un^{(2)}}.\end{align*} In this example we have been a little bit lazy: $\Eins_{\uF,\uB,\un}$ is not necessarily in $C(A)$, but we can extend the comultiplication in a natural way to all functions of the form $\Eins_{\uF,\uB,\un}$.
\end{ex}

\subsection{Proof of Theorem~\ref{thm-RH}}
We consider the products of functions of the form $\Eins_{\uF,\uB,\un}$ in $\mathcal H(A)$. By Theorem~\ref{thm-tree-band} and Lemmas~\ref{lem-stab-tree} and~\ref{lem-stab-band}, we only have to use the representation theory of quivers of type $A_l$ and $\widetilde A_{l-1}$, to calculate the Euler characteristics of the occurring varieties.\hfill $\Box$

 \subsection{Proof of Proposition~\ref{prop-RH-1}\eqref{part-RH-1}} We can use Theorem~\ref{thm-RH} to compute the product $\left(\Eins_{\uF,\uB,\un}\ast \Eins_{\uF',\uB',\un'}\right)(F_\ast(V))$ in $\mathcal H(A)$.  Let $Q$ be a tree, $V$ a $Q$-representation with $\dim_\CC(V_i)=1$ for all $i\in S_0$. So $V$ is a tree module. It is enough to compute $\left(\Eins_{\uF,\uB,\un}\ast \Eins_{\uF',\uB',\un'}\right)(V)$ in $\mathcal H(\CC Q)$. All sub- and factormodules of $V$ are again tree modules. If $\left(\Eins_{\uF,\uB,\un}\ast \Eins_{\uF',\uB',\un'}\right)(V)\neq0$, then $l(\uB)=l(\uB')=0$.\hfill $\Box$

\subsection{Proof of Proposition~\ref{prop-RH-1}\eqref{part-RH-2}}
 We use again Theorem~\ref{thm-RH}. So we can assume without loss of generality that $A=\CC Q$, where $Q$ is a quiver of type $\tilde A_{l-1}$.  All $Q$-modules are string or band modules $B'_\ast(V')$ such that $B'\colon Q\to Q$ is the identity winding.
 If $\left(\Eins_{\uF,\uB,\un}\ast \Eins_{\uF',\uB',\un'}\right)(V)\neq0$ with a band module $V$, then $l(\uB),l(\uB')\leq1$ holds by Remark~\ref{rem-eind-band}.

The equality $l(\uF)=l(\uF')$ can be shown by induction. Let $V$ be a band module and $U$ a submodule, which is isomorphic to a string module. It is enough to show that for the representation $V/U=(W_i,W_a)_{i\in Q_0,a\in Q_1}$ the equality \begin{align*}\dim(V/U)-1={\sum}_{a\in Q_1}\rk(W_a)\end{align*} holds, where $\rk(W_a)$ is the rank of the linear map $W_a$. This is clear since $V$ is a band and $U$ a string module with $\udim(U)\notin\ZZ(1,\ldots,1)$.\hfill $\Box$

\subsection{Proof of Proposition~\ref{prop-RH-2}\eqref{part-RH-3}} Let $M:=B_\ast(\lambda,m)$, $\pi\colon M\to B_\ast(\lambda,m-n)$ a projection, $K:=\bigoplus_iF^{(i)}_\ast(V)$ and $K':=\bigoplus_iF'^{(i)}_\ast(V)$. By Remark~\ref{rem-eind-band}, there exists a unique $U\subseteq B_\ast(\lambda,m)$ with $U\cong B_\ast(\lambda,n)$, so we can assume $B_\ast(\lambda,n)\subseteq B_\ast(\lambda,m-n')\subseteq B_\ast(\lambda,m)$. Define the varieties
\begin{align*}
 X:=&\Big\{U\subseteq M\Big|U\cong B_\ast(\lambda,n)\oplus K,M/U\cong  B_\ast(\lambda,n')\oplus K'\Big\}\\
 \overline X:=&\Big\{V\subseteq \overline M\Big|V\cong K,\overline M/V\cong K'\Big\}
\end{align*} with $\overline U:=\big(U\cap B_\ast(\lambda,m-n')\big)/ B_\ast(\lambda,n)$ for all $B_\ast(\lambda,n)\subseteq U\subseteq M$ and an algebraic morphism $\phi\colon X\to \overline X$ by $U\mapsto\overline U$. Using Remark~\ref{rem-eind-band} again,  $B_\ast(\lambda,n)\subseteq U\subseteq B_\ast(\lambda,m-n')$ for all $U\in X$. So $\phi$ is well-defined and injective.

 Let $V\in \overline X$. Since $V\cong K$ and  $\overline M/V\cong K'$ we have $B_\ast(\lambda,m-n')/\pi^{-1}(V)\cong K'$ and $M/B_\ast(\lambda,m-n')\cong B_\ast(\lambda,n')$. There exist two short exact sequences
$$0\to B_\ast(\lambda,n)\to\pi^{-1}(V)\to K\to 0$$
$$0\to K'\to M/\pi^{-1}(V)\to  B_\ast(\lambda,n')\to 0$$ Using Remark~\ref{rem-ART-Atilde}, we can assume without loss of generality that the direct summands of $K$ are preprojective $Q$-representations and the direct summands of $K'$ are preinjective ones. So both sequences split and this means that $\pi^{-1}(V)\cong B_\ast(\lambda,n)\oplus K$ and $M/\pi^{-1}(V)\cong K'\oplus B_\ast(\lambda,n')$. Thus $\pi^{-1}(V)\in X$ and $\overline{\pi^{-1}(V)}=V$. This shows that the Euler characteristics of both varieties are equal.\hfill $\Box$

\subsection{Proof of Proposition~\ref{prop-RH-2}\eqref{part-RH-4}} 

Proposition~\ref{prop-RH-2}\eqref{part-RH-4} follows inductively by the following lemma.
\begin{lem}\label{lem-prop-RH-2-4} Let $Q$, $m$, $B\colon Q\to Q$, $F\colon S\to Q$, $\uF(n)$, $\uF$ and $\uF'$ as in Proposition~\ref{prop-RH-2}\eqref{part-RH-4}. Let $M=B_\ast(\lambda,m)$ and $n\in\NN_{>0}$. Then
 \begin{align*}&\left(\Eins_{\uF(n)\dot\cup\uF}\ast \Eins_{\uF'}\right)(M)=
 \sum_{k\in\NN}\left(\left(\Eins_{\uF(n-1)}\otimes 1\right)\ast\Delta\left(\Eins_{\uF}\right)\ast\Delta\left(\Eins_{\uF'}\right)\right)(B_\ast(\lambda,m-k),I_k)
\end{align*} with $I_k$ is an indecomposable module and $\udim(I_k)=\udim(B_\ast(\lambda,k))-\udim(F_\ast(V))$.
\end{lem}

\begin{proof} Let $M=(M_i,M_a)_{i\in Q_0,a\in Q_1}$, $\uc:=\udim(F_\ast(V))$, $\ud^{(i)}:=\udim\left(F^{(i)}_\ast(V)\right)$ for all $i$ and $\ud=n\uc+\sum_i\ud^{(i)}$. By Remark~\ref{rem-ART-Atilde}, we know $\Eins_{\uF(n)}\ast \Eins_{\uF}=\Eins_{\uF(n)\dot\cup\uF}$. So we have to calculate the Euler characteristic of
\begin{align*}X=\Big\{(0\subseteq U\subseteq W\subseteq M)\in{\F}_{n\uc,\ud}(M)\Big|\Eins_{\uF(n)}(U)=\Eins_{\uF}(W/U)=\Eins_{\uF'}(M/W)=1\Big\}.\end{align*}
We can use now the arguments of the proof of Lemma~\ref{lem-band} in Section~\ref{sec-tree-band-2}:

 Let $\{e_{ik}|i\in Q_0,1\leq k\leq m\}$ be a basis of $M$ such that the following holds.\begin{enumerate}
 \item  For all $1\leq p\leq m$, the vector space $M^{(p)}:=\langle e_{i,k}|i\in Q_0,1\leq k\leq p\rangle$  is a subrepresentation of $M$ isomorphic to $B_\ast(\lambda,p)$.

 \item There exists a nilpotent endomorphism $\psi$ of $M$ such that $\psi(e_{i1})=0$ and $\psi(e_{ik})=e_{i,k-1}$ for all $1<k\leq m$ and all $i\in Q_0$.
 \end{enumerate}
The quiver $S$ is of type $A_{|\uc|}$ such that $S_0=\{1,\ldots,|\uc|\}$ and $S_1=\{s_1,\ldots,s_{|\uc|-1}\}$.

Let $(0\subseteq U\subseteq W\subseteq M)\in X$. Then $U\cong F_\ast(V)^n$. Using the Gauß algorithm, there exists a unique tuple $\uj(U)=(1\leq j_1<j_2<\ldots<j_n\leq m)$ as in Equation~\eqref{gl-band-1} and unique $\lambda_{kj}(U)\in\CC$ such that the vector space $U$ is spanned by \begin{align*}\left(M_{F_1(s_1)}^{\varepsilon_1}\ldots M_{F_1(s_q)}^{\varepsilon_q}\right)^{-1}\left(e_{F_0(1),j_p}+{\sum}_{j=1,j\neq j_k\forall k}^{j_p-1}\lambda_{pj}(U)e_{F_0(1),j}\right)\end{align*} with $1\leq p\leq n$ and $0\leq q< |\uc|$. This is well-defined since all linear maps $M_i$ are isomorphisms. The variety $X$ can be decomposed into a disjoint union of locally closed subsets \begin{align*}X_k:=\Big\{(U\subseteq W)\in X\Big|\left(\uj(U)\right)_1=k\Big\}.\end{align*} For each $k$ let \begin{align*}X_k^0:=\Big\{(U\subseteq W)\in X_k\Big|\lambda_{1j}(U)=0\ \forall j\Big\},\end{align*} a locally closed subset of $X$. Equation~\eqref{gl-band-2} defines again an algebraic morphism $\pi\colon X_k\to X_k^0$ with affine fibers.

For each $k$ there exists a $U_k\subseteq M$ such that $U_k\cong F_\ast(V)$, $U_k\subseteq U$ for all $(U\subseteq W)\in X_k^0$ and $M/U_k\cong M^{(m-k)}\oplus I_k$ with an indecomposable module $I_k$ as in the lemma. Since $|\uc|\geq|\ud^{(i)}|$ for all $i$ and $F_\ast(V)$ is preprojective, all sequences of the form \begin{align*}0\to F_\ast(V)\to\pi^{-1}(W)\to F_\ast(V)^{n-1}\oplus{\bigoplus}_iF^{(i)}_\ast(V)\to 0\end{align*} with a projection $\pi\colon M\to M/F_\ast(V)$ and a submodule $W\subseteq M/F_\ast(V)$ split. Let \begin{align*}\overline{X_k^0}:=\Big\{
 &\Big(U\subseteq W\subseteq M^{(m-k)}\Big),\Big(W'\subseteq I_k\Big)\Big|\\
 &\Eins_{\uF(n-1)}(U)=\Eins_{\uF}(W/U\oplus W')=\Eins_{\uF'}\left(M^{(m-k)}/W\oplus I_k/W'\right)=1\Big\}.
\end{align*} Using an R-grading, we can conclude, as in the proof of  Proposition~\ref{prop-RH-2}\eqref{part-RH-3}, that \begin{align*}\chi\left(X_k^0\right)=\chi\left(\overline{X_k^0}\right)=\left(\left(\Eins_{\uF(n-1)}\otimes 1\right)\ast\Delta\left(\Eins_{\uF}\right)\ast\Delta\left(\Eins_{\uF'}\right)\right)\left(M^{(m-k)},I_k\right)\end{align*} and \begin{align*}\chi\left(X\right)={\sum}_{k\in\NN}\left(\left(\Eins_{\uF(n-1)}\otimes 1\right)\ast\Delta\left(\Eins_{\uF}\right)\ast\Delta\left(\Eins_{\uF'}\right)\right)\left(M^{(m-k)},I_k\right).\end{align*}
\end{proof}

\section{String algebras}\label{sec-string-alg}

\subsection{Proof of Corollary~\ref{cor-RH-F-B}}
 If $A$ is a string algebra, then every indecomposable $A$-module is a string or a band module. So Corollary~\ref{cor-RH-F-B} follows directly from Proposition~\ref{prop-RH-1}, Lemma~\ref{lem-ind-enough} and Example~\ref{ex-ind-enough}.\hfill $\Box$

\subsection{Proof of Equation~\eqref{gl-RH-Kron}}  This can be proven by iterated use of Proposition~\ref{prop-RH-2}\eqref{part-RH-4}. We give here an alternative proof. By Section~\ref{sec-flag}, it is enough to show Equation~\eqref{gl-RH-Kron} for a string module $F_\ast(V)$ with dimension vector $\udim(B_\ast(\lambda,m))$. Using Theorem~\ref{thm-RH}, this can be computed by counting all orderings of the strings in $\uF$ and in $\uF'$.\hfill $\Box$

\subsection{Proof of Corollary~\ref{cor-RH-alg}}

 We use an induction over the dimension vectors of $Q$. Let $\ud$ be a dimension vector. Then the set \begin{align*}H_{\ud}:=\Big\{\Eins_{\uF,\uB,\un}\Big|\exists M\in\Rep(Q), \Eins_{\uF,\uB,\un}(M)\neq0, \udim(M)=\ud\Big\}\end{align*} is finite and the function $\Eins_{\ud}$ is the sum of all functions in $H_{\ud}$.

 It remains to show that each product $\Eins_{\uF,\uB,\un}\ast \Eins_{\uF',\uB',\un'}\in\mathcal H_{\ud}(A)$ is a linear combination of functions in $H_{\ud}$. Using Lemma~\ref{lem-ind-enough} and Example~\ref{ex-ind-enough}, we have to check that for all bands $B$ and $m\in\NN_{>0}$ the product \begin{align*}\Eins_{\uF,\uB,\un}\ast \Eins_{\uF',\uB',\un'}(B_\ast(\lambda,m))\end{align*} is independent of $\lambda\in\CC^\ast$. This is clear by Proposition~\ref{prop-RH-2}\eqref{part-RH-2} and an induction argument.\hfill $\Box$


\begin{thebibliography}{[LU]}
 \bibitem{ASS} I. Assem, D. Simson and A. Skowro\'nski, \textit{Elements of the representation theory of associative algebras, 1:\ Techniques of representation theory}, volume 65 of \textit{London Mathematical society Student Texts}, Cambridge University Press, cambridge, 2006.

 \bibitem{BilaBiru} A. Bialynicki-Birula, \textit{On fixed point schemes of actions of multiplicative and additive groups}, Topology 12, 99-102, 1973.

 \bibitem{Borel} A. Borel, \textit{Linear algebraic groups}, 2nd ed. Graduate Texts in Mathematics 126, Springer-Verlag, New York, 1991.

 \bibitem{BridgelandLaredo} T. Bridgeland and V. Taledano-Laredo, \textit{Stability conditions and Stokes factors}, \href{http://front.math.ucdavis.edu/0801.3974}{arXiv: 0801.3974}.

 \bibitem{CC} P. Caldero and F. Chapoton, \textit{Cluster algebras as Hall algebras of quiver representations}, Comment. Math. Helv., 81(3), 595-616, 2006.

 \bibitem{CK1} P. Caldero and B. Keller, \textit{From triangulated categories to cluster algebras}, Invent. Math., 172(1), 169-211, 2008.

 \bibitem{CK2} P. Caldero and B. Keller, \textit{From triangulated categories to cluster algebras. {II}}, Ann. Sci. \'Ecole Norm. Sup. (4), 39(6):983-1009, 2006.

 \bibitem{DWZII} H. Derksen, J. Weyman, and A. Zelevinsky, \textit{ Quivers with potentials and their representations II: Applications to cluster algebras}, 2009.

 \bibitem{FZI} S. Fomin and A. Zelevinsky, \textit{Cluster algebras. {I}. Foundations,} J. Amer. Math. Soc., 15(2), 497-529 (electronic), 2002.

 \bibitem{FZII} S. Fomin and A. Zelevinsky, \textit{Cluster algebras. {II}. Finite type classification}, Invent. Math., 154(1), 63-121, 2003.

 \bibitem{FZIV} S. Fomin and A. Zelevinsky, \textit{Cluster algebras. {IV}. Coefficients.}, Compos. Math., 143(1), 112-164, 2007.

\bibitem{Irelli} G.C. Irelli, \textit{Quiver Grassmannians associated with string modules}, \href{http://front.math.ucdavis.edu/0910.2592}{arXiv: 0910.2592}.

 \bibitem{Joyce} D. Joyce, \textit{Configurations in abelian categories. II. Ringel-Hall algebras}, Adv. Math. 210 (2007), 635-706.

 \bibitem{KapranovVasserot} M. Kapranov and E. Vasserot, \textit{Kleinian singularities, derived categories and Hall algebras}, Math. Ann. 316 (2000), 565-576.

\bibitem{Krause} H. Krause, \textit{Maps between tree and band modules}, J. Algebra 137 (1991), 186-194.

\bibitem{Lusztig} G. Lusztig, \textit{Quivers, perverse sheaves and quantized enveloping algebras}, J. Amer. Math. Soc. 4 (1991), 365-421.

\bibitem{Riedtmann} C. Riedtmann, \textit{Lie algebras generated by indecomposables}, J. Algebra 170 (1994), no. 2, 526-546.

\bibitem{Schiffmann} O. Schiffmann, \textit{Lectures of Hall algebras}, preprint, \href{http://front.math.ucdavis.edu/math/0611617}{arXiv: math/0611617}.

 \bibitem{Schof-general} A. Schofield, \textit{General representations of quivers}, Proc. London Math. Soc. (3) 65 (1992), 1, 46-64.

 \bibitem{Schofield} A. Schofield, \textit{Quivers and Kac-Moody Lie algebras}, unpublished manuscript.

 \bibitem{SS} D. Simson and A. Skowro\'nski, \textit{Elements of the representation theory of associative algebras, 2:\ Tubes and concealed algebras of euclidean type}, volume 71 of \textit{London Mathematical society Student Texts}, Cambridge University Press, cambridge, 2007.

\end{thebibliography}
\end{document}